\documentclass{amsart}
\usepackage{amsmath, amsthm, amsfonts, amssymb}
\usepackage{mathrsfs,mathtools}
\usepackage{microtype}
\usepackage[utf8]{inputenc}
\providecommand{\noopsort[1]{}}
\usepackage{bbm}
\usepackage{dsfont}
\usepackage{ifthen}
\usepackage{nicefrac}
\numberwithin{equation}{section}
\usepackage{a4}
\usepackage[active]{srcltx}
\usepackage{verbatim}
\usepackage{hyperref}
\DeclareMathOperator{\fix}{\mathrm{fix}}

\usepackage[shortlabels]{enumitem}
\setlist{leftmargin=*}
\setlist[1]{labelindent=1.2\parindent}

\newtheorem{thm}{Theorem}[section]
\newtheorem{cor}[thm]{Corollary}
\newtheorem{prop}[thm]{Proposition}
\newtheorem{lem}[thm]{Lemma}

\theoremstyle{remark}
\newtheorem{rem}[thm]{Remark}

\newtheorem{hyp}[thm]{Hypothesis}
\newtheorem{example}[thm]{Example}

\theoremstyle{definition}
\newtheorem{defn}[thm]{Definition}

\newcommand{\one}{\mathds{1}}

\newcommand{\R}{\mathds{R}}
\newcommand{\C}{\mathds{C}}

\newcommand{\N}{\mathds{N}}

\newcommand{\cL}{\mathscr{L}}
\newcommand{\cA}{\mathscr{A}}

\newcommand{\loc}{\mathrm{loc}}

\newcommand{\applied}[2]{\langle #1,#2\rangle}
\DeclarePairedDelimiter\norm{\lVert}{\rVert}
\DeclarePairedDelimiter\abs{\lvert}{\rvert}

\renewcommand{\Re}{{\mathrm{Re}}\,}
\renewcommand{\Im}{{\mathrm{Im}}\,}

\DeclareMathOperator{\rg}{\mathrm{rg}}
\newcommand{\eps}{\varepsilon}
\renewcommand{\phi}{\varphi}
\DeclareMathOperator{\lh}{\mathrm{span}}

\newcommand{\Dmax}{D_\mathrm{max}}
\newcommand{\weak}{\rightharpoonup}

\begin{document}
\title[Diffusion with nonlocal Dirichlet boundary conditions]{Diffusion with nonlocal Dirichlet boundary conditions on  unbounded domains}

\author{Markus C.\ Kunze}
\address{Markus C.\ Kunze\\Universit\"at Konstanz\\Fachbereich Mathematik und Statistik\\ Fach 193\\78457 Konstanz\\Germany}
\email{markus.kunze@uni-konstanz.de}

\begin{abstract}
We consider a second order differential operator $\cA$ on an open and Dirichlet regular set $\Omega\subset \R^d$ (which typically is unbounded) and subject to nonlocal Dirichlet boundary conditions of the form
\[
u(z) = \int_\Omega u(x)\mu (z, dx) \quad \mbox{ for } z\in \partial \Omega.
\]
Here, $\mu : \partial\Omega \to \mathscr{M}(\Omega)$ takes values in the probability measures on $\Omega$ and is continuous in the weak topoly $\sigma (\mathscr{M}(\Omega), C_b(\Omega))$. Under suitable assumptions on the coefficients in $\cA$, which may be unbounded, we prove that a realization $A_\mu$ of $\cA$ subject to the nonlocal boundary condition, generates a (not strongly continuous) semigroup on $L^\infty(\Omega)$. We establish a sufficient condition for this semigroup to be Markovian and prove that in this case, it enjoys the strong Feller property. We also study the asymptotic behavior of the semigroup.
\end{abstract}

\maketitle

\section{Introduction}

There is a well-known connection between Markov process on the one hand and parabolic partial differential equations and Markovian semigroups on the other hand. Starting with the seminal work of Feller \cite{feller-semigroup, feller-diffusion}, who studied the one-dimensional situation, this connection has developed into a rich and active field of scientific research. 
In this article, we seek to combine two aspects of this field which, over time, have received much attention: \emph{nonlocal boundary conditions} and \emph{unbounded coefficients}. 

We shall consider second order differential operators $\mathscr{A}$ on an  open subset $\Omega$ 
of $\R^d$, formally given by
\begin{equation}\label{eq.A}
\mathscr{A}u \coloneqq \sum_{i,j=1}^d a_{ij}D_iD_j u + \sum_{j=1}^db_j D_j u.
\end{equation}
In the typical applications we have in mind, the set $\Omega$ is unbounded and the coefficients $a_{ij}$ and $b_j$ are functions on $\Omega$ which may be unbounded as $|x|\to \infty$ within $\Omega$.
We will study a realization $A_\mu$ of $\mathscr{A}$ subject to nonlocal Dirichlet boundary conditions of the form
\begin{equation}
\label{eq.bc}
u(z) = \int_{\Omega} u(x) \mu (z, dx)
\end{equation}
for all $z\in \partial \Omega$.
In this equation, for every $z\in \Omega$ we are given a probability measure $\mu (z, \cdot)$ on $\Omega$. Nonlocal boundary conditions of this form arise naturally in applications, e.g.\ in financial mathematics (see \cite{gk02}), for the entropy in models of thermoelasticity (see \cite{day}), for heat conduction in ``well-stirred liquids'' (see \cite{berg}) or in the study of functional differential equations (see \cite{sku97}). 

This boundary condition has a clear probabilistic interpretation. Whenever a diffusing particle reaches the boundary of $\Omega$ in the point $z$, it immediately jumps back to the interior of $\Omega$. The point to which it jumps is chosen randomly according to the probability measure $\mu (z, \cdot)$. Thus, this boundary condition models what Feller in \cite{feller-diffusion} called an \emph{instantaneous return process}.\smallskip

On bounded domains, nonlocal boundary conditions of this form were considered by several authors, 
using different approaches  
\cite{akk16, b-ap07, b-ap09, galaskub, sk89, sk95,  taira16}. We should point out that this boundary condition falls in the so-called `non-transversal case' where the nonlocal term is of highest order in the boundary condition, since only terms of order 0 appear in the boundary condition. As a consequence, we cannot expect to obtain a strongly continuous semigroup on the space $C_b(\overline{\Omega})$ of bounded and continuous functions on $\overline{\Omega}$. Thus, to use strongly continuous semigroups, one has to either work on the $L^p$-scale (as was done in \cite{b-ap07, b-ap09}) or one has to consider a closed subspace (that heavily depends on the measure $\mu$) of the space of bounded and continuous functions (as was the case in \cite{galaskub, sk89, sk95}). The drawback of both approaches is that it is not clear how to extract transition probabilities from these semigroups. In \cite{akk16} we proved generation of an analytic semigroup on the space 
$L^\infty(\Omega)$. This semigroup is not strongly continuous but it enjoys the strong Feller property (so that in particular the semigroup is given through transition probabilities). We would like  to point out that in the case of nonlocal Robin boundary conditions (which, due to the presence of the normal derivative in the boundary condition which is of order 1, 
falls in the `transversal case') we do obtain strong continuity and analyticity of the semigroup on the space of bounded and continuous functions, see \cite{akk18}.\smallskip 

In contrast to the situation on \emph{bounded domains} we cannot expect analyticity of the semigroup for differential operators with unbounded coefficients on unbounded domains. This can be already seen in the prototype example of the Ornstein--Uhlenbeck semigroup, see \cite{dpl95}. Thus, in this article, one of the main obstacles to overcome is the choice of an appropriate semigroup setting, in which we can handle semigroups that are neither strongly continuous nor analytic. To that end, we will introduce the concept of a \emph{$*$-semigroup}, see
Section \ref{s.semigroup}. Even though these semigroups consist of adjoint operators they are a priori not adjoint semigroups in the sense of \cite{vanneerven1992}. This is due to the fact that the orbits need not be weak$^*$-continuous in $0$, so that  such a semigroup need not be the adjoint of a strongly continuous semigroup.
While the semigroups we will consider have no continuity at $0$, the regularity of the orbits for $t>0$ is quite good as a consequence of the strong Feller property, see Theorem \ref{t.sfcont}.
\smallskip

Our basic strategy to tackle the problem on unbounded domains is the same as in \cite{mpw02}, namely, we approximate the elliptic problem on unbounded domains by problems on bounded domains. This has been done in \cite{mpw02} for operators on all of $\R^d$. In the case of unbounded domains also Dirichlet (\cite{fmp04}) and Neumann (\cite{bf04, bfl07}) boundary conditions were considered. We should point out that in the cited articles the parabolic problem for $\mathscr{A}$ was treated independently of the elliptic problem, using heavily  Schauder theory for parabolic equations on bounded domains. 
However, in the Schauder approach to such problems higher regularity of the boundary and the coefficients is needed. Even worse for us, in Schauder boundary estimates also H\"older regularity of the boundary data is needed. In our situation, these boundary data are given via Equation \eqref{eq.bc}. If $u$ is continuous in the interior of $\Omega$, then the boundary data are also continuous. However, H\"older continuity cannot be expected. 

Thus, in this article we use a different approach which is abstract and, in spirit, is closer to semigroup theory in that we obtain all information about the parabolic equation by studying corresponding elliptic problem.
 Our main tool is a monotone convergence theorem for $*$-semigroups (Proposition \ref{prop:supremumsemigroup}).\smallskip 

Let us now specify our assumptions and state our main results. We refer to Section \ref{s.bdd} for unexplained terminology.

\begin{hyp}\label{h.coeff}
Throughout, let $\varnothing\neq \Omega \subset \R^d$ be an open and Dirichlet regular set. Concerning the coefficients in Equation \eqref{eq.A}, we assume that $a_{ij} \in C(\overline{\Omega})$, $b_j \in L^\infty_{\mathrm{loc}}(\overline{\Omega})$ are real valued for $i,j=1, \ldots, d$. The diffusion coefficients $a_{ij}$ are assumed to be symmetric (i.e. $a_{ij}=a_{ji}$ for $i,j=1, \ldots, d$) and strictly elliptic in the sense that there is a function $\eta \in C(\overline{\Omega})$ with $\eta(x) >0$ for every $x\in \overline{\Omega}$, such that for all $\xi \in \R^d$ we have
\[
\sum_{i,j=1}^d a_{ij}(x)\xi_i\xi_j \geq \eta (x)|\xi|^2
\]
for every $x\in \overline{\Omega}$. In addition, we assume that either
\begin{enumerate}
[(i)]
\item the coefficients $a_{ij}$ are locally Dini continuous for $i,j=1, \ldots, d$ or
\item for every $n\in \N$ the set $\Omega \cap B_n(0)$ satisfies the uniform exterior cone condition. Here, $B_n(0)$ denotes the Euclidean ball of radius $n$, centered at $0$.
\end{enumerate} 
\end{hyp}

In the above, $L^\infty_{\mathrm{loc}}(\overline{\Omega})$ refers to the space of all functions that are essentially bounded on compact subsets of $\overline{\Omega}$. Thus the drift coefficients $b_j$ and the diffusion coefficients $a_{ij}$ may be unbounded as $|x|\to \infty$, but they may \emph{not} explode near the boundary $\partial \Omega$. Likewise, the ellipticity constant $\eta$ may degenerate to $0$ as $|x|\to \infty$, but not near the boundary. 

Next, we make our assumptions concerning the boundary condition precise. We denote the Borel $\sigma$-algebra on $\Omega$ by $\mathscr{B}(\Omega)$ and the space of (signed) Borel measures on $\Omega$ by $\mathscr{M}(\Omega)$. 

\begin{hyp}\label{h.mu}
We let $\mu : \partial\Omega\times \mathscr{B}(\Omega) \to \mathscr{M}(\Omega)$. We will sometimes write
$\mu(z) \coloneqq \mu (z, \cdot) \in \mathscr{M}(\Omega)$. We assume that
\begin{enumerate}
[(i)]
\item $\mu(z)$ is a probability measure for every $z\in \partial\Omega$ and 
\item the map $z\mapsto \mu(z)$ is $\sigma (\mathscr{M}(\Omega), C_b(\Omega))$-continuous.
\end{enumerate}
\end{hyp}
Here, $\sigma (\mathscr{M}(\Omega), C_b(\Omega))$ refers to the weak topology on $\mathscr{M}(\Omega)$ induced
by the bounded and continuous functions. Thus, condition (ii) is equivalent to asking that the map
\[
z \mapsto \int_\Omega f\, d\mu (z)
\]
is continuous for every function $f\in C_b(\Omega)$.

As in \cite{akk16}, given an open set $U\subset \R^d$, we set
\[
W(U)\coloneqq \bigcap_{1<p<\infty} W^{2,p}_{\mathrm{loc}}(U).
\]
By elliptic regularity, see \cite[Lemma 9.16]{gt}, we have $u\in W(\Omega)$ whenever $u\in W^{2,p}_{\mathrm{loc}}(\Omega)$ for some
$1<p<\infty$ and $\mathscr{A}u \in L^\infty_{\mathrm{loc}}(\Omega)$. We now complement our differential operator $\mathscr{A}$ with nonlocal boundary conditions of the form \eqref{eq.bc}. To that end,
we define the \emph{maximal domain} $\Dmax$ by
\begin{align*}
D_{\mathrm{max}}  \coloneqq \Big\{ u&  \in C_b(\overline{\Omega})\cap W(\Omega) : \cA u \in L^\infty(\Omega)\\
&\qquad u(z) = \int_{\Omega} u(x) \mu (z, dx)\,\, \forall\, z \in \partial\Omega\Big\}.
\end{align*}

Our main result is as follows.

\begin{thm}\label{t.m1}
Assume Hypotheses \ref{h.coeff} and \ref{h.mu}. Then there is a subspace $D(A_\mu)$ of $D_{\mathrm{max}}$, such that
the operator $A_\mu : D(A_\mu) \to L^\infty(\Omega)$, $A_\mu u = \mathscr{A} u$ has the following properties:
\begin{enumerate}
[(a)]
\item $(0,\infty) \subset \rho (A_\mu)$ and $R(\lambda, A_\mu)$ is a positive operator on $L^\infty(\Omega)$ which satisfies $\|\lambda R(\lambda, A_\mu)\|\leq 1$ for all $\lambda >0$;
\item For every $\lambda >0$ and $f\in L^\infty(\Omega)_+$ the function $u \coloneqq R(\lambda, A_\mu)f$ is the smallest positive solution of the equation $\lambda u -\mathscr{A}u =f$ in $D_\mathrm{max}$.
\item $A_\mu$ generates a positive and contractive $*$-semigroup $T_\mu = (T_\mu (t))_{t\geq 0}$ on $L^\infty(\Omega)$.
\item $D(A_\mu) = D_{\mathrm{max}}$ if and only if $\one \in D(A_\mu)$. In this case the semigroup $T_\mu$ enjoys the strong Feller property.
\item If $\ker A_\mu = \lh\{\mathds{1}\}$ then there is at most one invariant probability measure for the semigroup $T_\mu$. If there is an invariant probability measure $\nu^\star$, then for every $f\in L^\infty(\Omega)$ we have
\[
\lim_{t\to\infty} T_\mu(t)f = \int_{\overline{\Omega}} f\, d\nu^\star \cdot\mathds{1}
\]
uniformly on compact subsets of $\overline{\Omega}$ whereas for the adjoint semigroup $T_\mu'$ on the space $\mathscr{M}(\overline{\Omega})$ we have for every $\nu \in \mathscr{M}(\overline{\Omega})$
\[
\lim_{t\to\infty} T_\mu'(t)\nu = \nu(\overline{\Omega})\nu^\star
\]
in total variation norm.
\end{enumerate}
\end{thm}

As we are dealing with elliptic equations with unbounded coefficients, we cannot expect uniqueness for the solution of the associated elliptic equation in general so that we may have several solution of the elliptic equation $\lambda u - \mathscr{A}u$ in $D_\mathrm{max}$. As $A_\mu$ is a bijection between $D(A_\mu)$ and $L^\infty(\Omega)$, part (d) of Theorem \ref{t.m1} characterizes unique solvability. As is to be expected, 
we can establish this unique solvability making use of an appropriate Lyapunov function, see Corollary \ref{cor:uniquesolution}. We should point out that our assumptions on the Lyapunov function in Corollary  \ref{cor:uniquesolution} do not involve the boundary condition (though we have to additionally assume a weak concentration assumption on the measures $\mu$) so that Lyapunov functions can be constructed as in \cite{mpw02}, imposing suitable growth conditions on the coefficients.

Lyapunov functions can also be used to establish existence of an invariant measure. However, typically the assumptions on such a Lyapunov function are more restrictive then in the case where we merely want to establish uniqueness for the elliptic equation. In our situation, we need to involve the boundary condition in our requirements on the Lyapunov function to ensure existence of an invariant measure, see Theorem \ref{t.lyapunov}.

In Section \ref{sect.example}, we present concrete examples where we can construct Lyapunov functions and thus establish existence of an invariant probability measure. In these examples, $\Omega$ is an outer domain and the differential operator $\mathscr{A}$ has coefficients which grow polynomially.
\smallskip

This article is organized as follows. In Section \ref{s.semigroup} we introduce the notion of a $*$-semigroup on the dual of a separable Banach space and prove some results that will be used later on. Section \ref{s.sf} is concerned with the notion of `kernel operator' and the strong Feller property. These two sections might also be of independent interest and are presented in an abstract framework. After recalling some results concerning diffusions with nonlocal boundary conditions on bounded domains in Section \ref{s.bdd}, we study the elliptic equation $\lambda u - \cA u = f$ in Section \ref{s.ell}. There we prove parts (a) and (b) of Theorem \ref{t.m1}. In Section \ref{s.lyap} we address the unique solvability of the elliptic equation. Parts (c) and (d) of Theorem \ref{t.m1} are proved in Section \ref{s.tmu} and our results concerning the asymptotic behavior of the semigroup (in particular the proof of part (e) of Theorem \ref{t.m1}) are found in Section \ref{s.asympt}. In the concluding Section \ref{sect.example}, we present our examples.

\section{Semigroups on the dual of a separable space}\label{s.semigroup}

As already mentioned, we will consider semigroups on the space $L^\infty(\Omega)$ in subsequent sections. It follows from a result of Lotz \cite{lotz1985} that a strongly continuous semigroup on $L^\infty(\Omega)$ is already norm continuous and thus has a bounded generator. To handle semigroups that are not strongly continuous, we will introduce the notion of a $*$-semigroup. At first, the only structural property of $L^\infty(\Omega)$ that we will use is that it is the dual space of the separable space $L^1(\Omega)$. We have therefore decided to treat semigroups on the dual of a separable space in general, as the results obtained here might also be of interest in other situations.  We should also mention that some of the results presented here can be obtained from the more general theory of ``semigroups on norming dual pairs'', see \cite{kunze2009,kunze2011}. However, the situation of a dual space is  easier to handle and proofs simplify. Thus, for the convenience of the reader, we will give a self-contained exposition and complete proofs.
\smallskip

Throughout this section, $X$ denotes a separable Banach space and $X^*$ its dual space. If $T$ is an adjoint operator, say
$T= S^*$ for some bounded linear operator $S\in \mathscr{L}(X)$, then clearly $T$ is a bounded operator on $X^*$ which is also weak$^*$-continuous. Conversely, if $T$ is a weak$^*$-continuous linear map on $X^*$, then $T$ is an adjoint operator, thus in particular bounded. To shorten notation, we write $\sigma^* \coloneqq \sigma (X^*, X)$ for the weak$^*$-topology on $X^*$ and $\cL(X^*,\sigma^*)$ for the space of weak$^*$-continuous operators on $X^*$.

\begin{lem}
	\label{lem:seqweakstarcont}
	Let $X$ be a separable Banach space and $T\colon X^* \to X^*$ be a bounded linear operator.
	 Then $T$ is weak$^*$-continuous if and only if $T$ is sequentially weak$^*$-continuous.
\end{lem}
\begin{proof}
	Clearly, every continuous mapping is sequentially continuous. So assume that $T$ is sequentially weak$^*$-continuous.
	By definition of the weak$^*$-topology it suffices to show that for every $x\in X$ 
	the linear mapping $\phi_x \colon X^* \to \R$, given by $\phi_x(x^*) \coloneqq \applied{Tx^*}{x}$, is continuous. This, in turn,  is equivalent
	to $\ker \phi_x$ being weak$^*$-closed. By the Krein-\v{S}mulian theorem, it suffices to show that
	$\phi_x \cap \overline{B}_r(0)$ is closed for each $x\in X$ and $r>0$, where $\overline{B}_r(0)$ denotes the norm-closed ball of radius $r>0$ in $X^*$.
	As $X$ is separable, the weak$^*$-topology is metrizable on norm-bounded sets, whence it suffices to check that
	$\phi_x \cap \overline{B}_r(0)$ is sequentially closed for each $x\in X$ and $r>0$. This, however, follows immediately from our assumption since each
	$\phi_x$ is sequentially weak$^*$-continuous.
\end{proof}

\begin{defn}\label{d.sg}
	Let $T=(T(t))_{t > 0} \subseteq \cL(X^*,\sigma^*)$ be a family of operators 
	such that $T(t+s)  =T(t)T(s)$ for all $t,s>0$  and that for all $x^* \in X^*$ and $x\in X$ the mapping $t\mapsto \applied{T(t)x^*}{x}$ is measurable.
	Then $T$ is called a \emph{$*$-semigroup on $X^*$}. If $\norm{T(t)}\leq 1$ for all $t>0$, then $T$ is called \emph{contractive}.
	Moreover, $T$ is said to be \emph{injective} if $T(t)x=0$ for all $t>0$ implies that $x=0$.
\end{defn}

Next, for a contractive $*$-semigroup $T=(T(t))_{t>0}$ and $\Re \lambda>0$ we define the operator $R(\lambda)$ on $X^*$ by
\begin{align}
\label{eqn:Rlambdadef}
\applied{R(\lambda)x^*}{x} \coloneqq \int_0^\infty e^{-\lambda t} \applied{T(t)x^*}{x} \,dt,
\end{align}
i.e.\ $R(\lambda)$ is the \emph{Laplace transform} of $t\mapsto T(t)x^*$, computed by means of the weak$^*$-integral. This is well-defined as the right-hand side
of \eqref{eqn:Rlambdadef} defines a bounded linear functional on $X$ in view of the boundedness of $T$.

\begin{prop}
\label{prop:pseudoresolvent}
	For a contractive $*$-semigroup $T=(T(t))_{t>0}$, the following assertions hold.
	\begin{enumerate}[(a)]
	\item $R(\lambda) \in \cL(X^*,\sigma^*)$ for all $\Re \lambda>0$.
	\item For $\lambda,\mu \in \{z\in \C : \Re z > 0\}$ we have 
	\begin{equation}\label{eq.reseq} 
	R(\lambda) - R(\mu) = (\mu-\lambda)R(\lambda)R(\mu),
	\end{equation}
	i.e.\ $R(\lambda)_{\Re \lambda>0}$ is a \emph{pseudoresolvent}.
	\end{enumerate}
\end{prop}
\begin{proof}
(a) In view of Lemma \ref{lem:seqweakstarcont}, this follows immediately from Equation \eqref{eqn:Rlambdadef} and the dominated convergence theorem.

(b) We first show that for each $\Re \lambda >0$, $x^* \in X^*$ and $h>0$
\begin{align}
\label{eqn:reseqproof}
	\int_0^h e^{-\lambda t} T(t)x^* \, dt = R(\lambda) x^* - e^{-\lambda h}T(h)R(\lambda)x^*,
\end{align}
where the integral on the left-hand side is to be understood as a weak$^*$-integral as before. To see this, fix $x\in X$ and let $S(t) \in \cL(X)$ be
such that $S(t)^* =T(t)$. Then
\begin{align*}
	\applied{T(h)R(\lambda)x^*}{x} &= \applied{R(\lambda)x^*}{S(h)x} = \int_0^\infty e^{-\lambda t}\applied{T(t)x^*}{S(h)x}\, dt\\
	&= \int_0^\infty e^{-\lambda t} \applied{T(t+h)x^*}{x} \, dt = e^{ht} \int_h^\infty e^{-\lambda t}\applied{T(t)x^*}{x} \, dt\\
	&= e^{\lambda h} \biggl( \applied{R(\lambda)x^*}{x} - \int_0^h e^{-\lambda r}\applied{T(r)x^*}{x} \,d r\biggr).
\end{align*}
As $x\in X$ was arbitrary, Equation \eqref{eqn:reseqproof} is proved. Now let $0 < \Re \mu < \Re \lambda$. Then we have that
\begin{align*}
	(\mu-\lambda) R(\lambda)R(\mu)x^* &= (\mu-\lambda) \int_0^\infty e^{-\lambda t}T(t)R(\mu)x^* \, d t\\
	&= (\mu-\lambda) \int_0^\infty e^{(\mu-\lambda)t} e^{-\mu t}T(t)R(\mu)x^* \, d t\\
	&= (\mu-\lambda)\int_0^\infty e^{(\mu-\lambda)t} \biggl( R(\mu)x^* - \int_0^te^{-\mu r} T(r)x^* \,d r\biggr) \,d t\\
	&= -R(\mu)x^* - (\mu-\lambda) \int_0^\infty e^{-\mu r} T(r)x^* \int_r^\infty e^{(\mu-\lambda)t}\,d t \,d r\\
	&= R(\lambda)x^* - R(\mu)x^*.
\end{align*}
Here, the third equality uses Equation \eqref{eqn:reseqproof}, the fourth Fubini's theorem and the inequality $\Re \mu < \Re \lambda$. Of course, each integral in this calculation
is to be understood in the weak$^*$-sense.

Finally, let us consider the situation that $0< \Re\mu = \Re \lambda$. We set $\lambda_n \coloneqq \lambda + n^{-1}$, so that $0<\Re\mu < \Re\lambda_n$. It follows from Equation \eqref{eqn:Rlambdadef} and dominated convergence, that $R(\lambda_n)x^*$ converges in the weak$^*$-sense to $R(\lambda)x^*$ for every $x^*\in X^*$. By the above, we have
\[
R(\lambda_n)x^*-R(\mu)x^* = (\mu-\lambda_n)R(\lambda_n)R(\mu)x^*
\]
for every $x^*\in X^*$ and $n\in \N$. Upon $n\to \infty$ we obtain \eqref{eq.reseq}.
\end{proof}

Our next goal is to prove that the Laplace transform $(R(\lambda))_{\Re \lambda > 0}$ determines the semigroup $(T(t))_{t>0}$ uniquely.
To this end, we use the following Lemma, taken from \cite[Lemma 3.16.5]{arendt2001}.

\begin{lem}
\label{lem:nullsetempty}
	Let $N\subseteq (0,\infty)$ be a set of Lebesgue measure $0$ and assume that $t,s \not\in N$ implies $t+s \not\in N$. Then $N=\varnothing$.
\end{lem}

\begin{thm}
\label{thm:laplacetransformunique}
	Let $T_1=(T_1(t))_{t>0}$ and $T_2=(T_2(t))_{t>0}$ be contractive $*$-semigroups on $X^*$ with Laplace transforms $(R_1(\lambda))_{\Re \lambda>0}$ and
	$(R_2(\lambda))_{\Re \lambda>0}$, respectively. If there exists $\lambda_0\geq 0$ such that $R_1(\lambda) = R_2(\lambda)$ for all $\lambda > \lambda_0$,
	then $T_1=T_2$.
\end{thm}
\begin{proof}
	Let $\lambda_0\geq 0$ such that 
	\[ \int_0^\infty e^{-\lambda t}\applied{T_1(t)x^*}{x} \,d t = \int_0^\infty e^{-\lambda t} \applied{T_2(t)x^*}{x} \,d t\]
	for all $x^* \in X^*$, $x\in X$  and $\lambda > \lambda_0$. By the uniqueness theorem for Laplace transforms \cite[Theorem 1.7.3]{arendt2001},
	there is a null set $N(x^*,x)$ such that
	\[ \applied{T_1(t)x^*}{x} = \applied{T_2(t)x^*}{x} \text{ for all } t\not\in N(x^*,x).\]
	Now pick a dense sequence $(x_n)\subseteq X$ and define $N(x^*) \coloneqq \bigcup_{n\in\N} N(x^*,x_n)$. Then each $N(x^*)$ is a null set
	and $T_1(t)x^*  =T_2(t)x^*$ for all $t\not\in N(x^*)$ as $(x_n)$ separates the points of $X^*$.
	Since $X$ is separable, we may find a norming sequence $(x_n^*)\subseteq X^*$ 
	and put $N \coloneqq \bigcup_{n\in\N} N(x_n^*)$.
	Let $S_1(t),S_2(t) \in \cL(X)$ be such that $S_1^*(t)=T_1(t)$ and $S_2^*(t) = T_2(t)$ for all $t>0$.
	Since 
	\[ \applied{x_n^*}{S_1(t)x} = \applied{T_1(t)x_n^*}{x} = \applied{T_2(t)x_n^*}{x} = \applied{x_n^*}{S_2(t)x}\]
	for all $t\not \in N$, $x\in X$ and $n\in\N$ and since the norming set $\{x_n^*: n\in\N\}$ separates the points in $X$,
	we have $S_1(t) = S_2(t)$ and thus also $T_1(t)=T_2(t)$ for all $t\not\in N$.
	Now consider $M\coloneqq \{ t>0 :T_1(t)\neq T_2(t)\}$. Then $M\subseteq N$ is a null set and it follows from the semigroup law that
	$t,s \not\in M$ implies that $t+s \not\in M$. Hence, $M=\varnothing$ by Lemma \ref{lem:nullsetempty}.
\end{proof}

\begin{example}
Without the assumption that $X$ be separable, the Laplace transform does not determine the semigroup uniquely, even if $X$ is a Hilbert space. Indeed, consider counting measure $\zeta$ on $\R$. The corresponding $L^2$-space is $\ell^2(\R)$ and consists of functions of the form $f(x) = \sum \alpha_n\mathds{1}_{\{x_n\}}$, where $(x_n)$ is a sequence of real numbers and $(\alpha_n)$ is a square-summable sequence. Now consider the shift semigroup $T=(T(t))_{t>0}$, given by $T(t)f(x) = f(x+t)$.
Then, given $f,g\in \ell^2(\R)$, we have $\langle T(t)f, g\rangle = 0$, except for at most countably many values of $t$. Consequently, the Laplace transform is given by $R(\lambda) \equiv 0$, whereas the semigroup is \emph{not} the zero semigroup.
\end{example}

Next, we want to associate a generator to a $*$-semigroup, i.e.\ an operator such that the resolvent of that operator is given as the Laplace transform of the semigroup. However, in order to do so, the Laplace transform has to 
consist of injective operators, which is not always the case. 

Since $(R(\lambda))_{\Re \lambda>0}$ is a pseudoresolvent by Proposition \ref{prop:pseudoresolvent}, the kernel $\ker R(\lambda)$  for $\lambda \in \C_+ \coloneqq \{ z\in \C : \Re z >0\}$ is independent of $\lambda$. Moreover, if $\ker R(\lambda)=\{0\}$ for some/all $\lambda \in \C_+$, then there exists an operator $A$ with $\C_+ \subseteq \rho(A)$ and $R(\lambda,A) =R(\lambda)$ for all $\lambda \in \C_+$, see \cite[Proposition B.6]{arendt2001}. In that case, $D(A) = \rg R(\lambda)$ and $A=\lambda -R(\lambda)^{-1}$.
	The proof of Theorem \ref{thm:laplacetransformunique} shows that $\ker R(\lambda)=\{0\}$ for some/all $\lambda \in \C_+$ if and only if the semigroup $T$ is injective. We may thus define:

\begin{defn}
Let $T=(T(t))_{t>0} \subseteq \cL(X^*,\sigma^*)$ be an injective  and contractive $*$-semigroup. The unique operator $A$ such that
\[ R(\lambda,A) x^* = \int_0^\infty e^{-\lambda t} T(t)x^* \,d t \]
for all $\Re \lambda >0$ is called the \emph{generator of $T$}.
\end{defn}

We can now characterize the generator of an injective and contractive $*$-semigroup as follows:

\begin{prop}
\label{prop:AWF}
	Let $T=(T(t))_{t>0} \subseteq \cL(X^*,\sigma^*)$ be an injective and contractive $*$-semigroup with generator $A$. Then for all $y^*,z^* \in X^*$ the following are equivalent:
	\begin{enumerate}[(i)]
	\item $y^* \in D(A)$ and $Ay^* = z^*$.
	\item $\int_0^t T(s)z^* \,d s = T(t)y^*-y^*$ for all $t>0$.
	\end{enumerate}
\end{prop}
\begin{proof}
	(i) $\Rightarrow$ (ii): For fixed $t>0$ and $x\in X$ we define the holomorphic functions $f,g \colon \C \to \C$ by
	\begin{align*}
	f(\lambda) &\coloneqq \lambda\int_0^t e^{-\lambda s}\applied{T(s)y^*}{x} \,d s - \int_0^t e^{-\lambda s}\applied{T(s)z^*}{x} \,d s \\
	g(\lambda) &\coloneqq \applied{y^*}{x} - e^{-\lambda t}\applied{T(t)y^*}{x}.
	\end{align*}
	For $\Re \lambda>0$, it follows from Equation \eqref{eqn:reseqproof} with $x^* = \lambda y^* -z^*$ that
	\[ \int_0^t e^{-\lambda s}T(s)x^* \,d s = R(\lambda)x^* - e^{-\lambda t}T(t)R(\lambda)x^* =y^* -e^{-\lambda t}T(t)y^*\]
	since $R(\lambda )x^* = y^*$. This shows that $f(\lambda) = g(\lambda)$ for all $\Re \lambda>0$ and thus, by the uniqueness theorem
	for holomorphic functions, for all $\lambda \in \C$. In particular, $f(0) = g(0)$ and this implies (ii) as $x\in X$ was arbitrary.

	(ii) $\Rightarrow$ (i): If (ii) holds, it follows from Fubini's theorem that
	\begin{align*}
	\lambda R(\lambda)y^* - y^* &= \int_0^\infty \lambda e^{-\lambda t}\big(T(t)y^* - y^*\big) \,d t = \int_0^\infty \lambda e^{-\lambda t} \int_0^t T(s)z^* \,d s \,d t \\
	&= \int_0^\infty \int_s^\infty \lambda e^{-\lambda t}T(s)z^* \, dt\, ds = \int_0^\infty e^{-\lambda s}T(s)z^*\, ds = R(\lambda)z^*.
	\end{align*}
	This shows that $y^* = R(\lambda)(\lambda y^* -z^*) \in D(A)$ and $Ay^* = z^*$.
\end{proof}

Recall that the semigroups we consider here are not strongly continuous (not even weak$^*$-continuous). Nevertheless, 
we can expect  some continuity of the orbits.

\begin{cor}\label{c.strongcont}
Let $T= (T(t))_{t>0} \subset \mathscr{L}(X^*, \sigma^*)$ be an injective and contractive $*$-semigroup with generator $A$. Then for $x^*\in \overline{D(A)}$ the orbit $t\mapsto T(t)x^*$ is $\|\cdot\|$-continuous. In particular, $T$ is strongly continuous if and only if $\overline{D(A)} = X^*$.
\end{cor}

\begin{proof}
For $x^* \in D(A)$ we have, as a consequence of Proposition \ref{prop:AWF}, that
\[
\|T(t)x^*-T(s)x^*\| = \Big\|\int_s^t T(r)Ax^*\, dr\Big\| \leq |t-s| \|Ax^*\| \to 0
\]
as $t\to s$. Making use of the uniform boundedness of the operators, a $3\eps$-argument shows that this remains true for $x^* \in \overline{D(A)}$.
\end{proof}

We now add an additional structure to our space $X$, namely, we assume that $X$ is a Banach lattice. We denote the positive cone of $X$ by $X_+$. The dual cone in $X^*$ is denoted by $X^*_+$. Note that we have $x^* \in X^*_+$ if and only if 
$\langle x^*,x\rangle \geq 0$ for all $x\in X_+$. An operator $T$ on $X^*$ is called \emph{positive} if $Tx^* \in X^*_+$
whenever $x^*\in X^*_+$. This notion defines an ordering on $\mathscr{L}(X^*)$ by setting $T_1 \leq T_2$ if and only if $T_2-T_1 \geq 0$. We call a $*$-semigroup $(T(t))_{t>0}$ \emph{positive} if $T(t)$ is a positive operator for every $t>0$.

\begin{prop}
\label{prop:positivesemigroups}
Let $X$ be a separable Banach lattice and $T_1=(T_1(t))_{t>0}$ and $T_2=(T_2(t))_{t\geq 0}$ be contractive $*$-semigroups
on $X^*$ with Laplace transforms $(R_1(\lambda))_{\Re \lambda>0}$ and $(R_2(\lambda))_{\Re \lambda>0}$, respectively,
and suppose that $T_1$ is positive.  

Then $T_1(t) \leq T_2(t)$ for all $t>0$ if and only if there exists $\lambda_0 \geq 0$ such that 
$R_1(\lambda) \leq R_2(\lambda)$ for all $\lambda > \lambda_0$.
In particular, $T_2$ is positive if and only if $R_2(\lambda)$ is positive for all real $\lambda$ large enough.
\end{prop}
\begin{proof}
	If $T_1(t) \leq T_2(t)$ for all $t>0$, then clearly $R_1(\lambda) \leq R_2(\lambda)$ for all $\lambda>0$. Now suppose that $R_1(\lambda) \leq R_2(\lambda)$
	holds for all $\lambda > \lambda_0$.
	Let $x\in X_+$, $x^* \in X^*_+$ and define $r_{x,x^*} \colon (\lambda_0,\infty) \to [0,\infty)$ by
	\[ r_{x,x^*}(\lambda) \coloneqq \applied{R_2(\lambda)x^* - R_1(\lambda)x^*}{x} 
	= \int_0^\infty e^{-\lambda t}\applied{T_2(t)x^*-T_1(t)x^*}{x} \, dt.\]
	It follows from the resolvent equation \eqref{eq.reseq}, 
	that $r_{x,x^*}$ is infinitely many times differentiable with
	\[ \frac{d^n}{d \lambda^n} r_{x,x^*}(\lambda) = (-1)^n n!\applied{R^{n+1}_2(\lambda)x^* - R^{n+1}_1(\lambda)x^*}{x}.\]
	Now the Post--Widder inversion formula \cite[Proposition 1.7.7] {arendt2001} implies that there is a null set $N(x^*, x) \subset (0,\infty)$ such that
	\[ \applied{T_2(t)x^* - T_1(t)x^*}{x} =\lim_{n\to \infty} (-1)^n \frac{1}{n!} \Big( \frac{n}{t}\Big)^{n+1}\frac{d^n}{d \lambda^n}r_{x,x^*}\Big(\frac{n}{t}\Big)\geq 0\] 
	for all $t\in (0,\infty)\setminus N(x^*,x)$.
	Now we proceed similarly to the proof of Theorem \ref{thm:laplacetransformunique}.
	Since $X$ is separable, we find a sequence $(x_n) \subseteq X_+$  which is dense in the positive cone $X_+$ 
	and a sequence $(x_n^*)\subseteq X^*_+$ which is norming in $X$. 
	Set $N\coloneqq \bigcup_{n,m\in \N} N(x_n^*,x_m)$.
	Since $\{x_n^* : n\in\N\}$ is weak$^*$-dense in $X^*_+$, it follows from
	\[ \applied{T_2(t)x_n^*-T_1(t)x_n^*}{x_m}\geq 0\]
	for all $n,m\in\N$ and $t\in (0,\infty)\setminus N$,
	that $T_2(t)- T_1(t)$ is positive for all $t\in (0,\infty)\setminus N$.
	Now consider $M\coloneqq \{ t>0 : T_1(t) \not\leq T_2(t)\}$, which is contained in $N$ and thus a null set.
	Moreover, for $t,s \not\in M$ it follows from the positivity of $T_1$ that
	\[ T_1(t+s) = T_1(t)T_1(s) \leq T_1(t) T_2(s) \leq T_2(t)T_2(s) = T_2(t+s),\]
	i.e.\ $t+s\not\in M$.
    Thus, Lemma \ref{lem:nullsetempty} implies that $T_1(t)\leq T_2(t)$ for all $t\in (0,\infty)$.
\end{proof}

Recall that a Banach lattice is called a \emph{KB-space} if every increasing and norm-bounded net of positive vectors converges in norm,
cf.\ \cite[Definition 2.4.11]{meyer1991}.  For instance, every $L^1$-space has this property.

\begin{lem}\label{l.increase}
Let $X$ be a separable KB-space and let $(U_n) \subset \cL(X^*, \sigma^*)$ be an increasing sequence of positive and contractive
operators.
\begin{enumerate}
\item There is a positive contraction $U \in \cL (X^*, \sigma^*)$ such that $Ux^* = \sup_{n\in\N} U_nx^*$ for 
all $x^*\in X^*_+$. We write $U_n \uparrow U$.
\item If $(V_n) \subset \cL (X^*, \sigma^*)$ is another increasing sequence of positive contractions with $V_n \uparrow V$, then
$U_nV_n\uparrow UV$.
\end{enumerate}
\end{lem}

\begin{proof}
(1) Pick $S_n \in \cL (X)$ such that $S_n^* = U_n$. For $x \in X_+$ the sequence $S_nx$ is increasing and norm-bounded. Since $X$ is a KB-space,  the limit $\tilde S x \coloneqq S_n x$ exists. Obviously, $\tilde S$ is additive and positively homogeneous on the positive
cone $X_+$. Consequently, it uniquely extends to a positive linear operator $S$ on $X$, cf.\ \cite[Lemma 1.3.3]{meyer1991}.
It follows that in the ordering of $\mathscr{L}(X)$ we have $S= \sup_{n\in\N} S_n$ and hence $U \coloneqq S^* = \sup_{n\in\N} S_n^* = \sup_{n\in\N} U_n$ is 
an adjoint operator. That $U$ is a positive contraction is obvious.

(2) Clearly, $U_nV_n \leq UV$  and hence $\sup_{n\in\N} U_nV_nx^* \leq UVx^*$ for all $x^* \in X^*_+$. On the other hand,
for fixed $m\in \N$ we have
\[
\sup_{n\in\N} U_nV_nx^* \geq \sup_{n\in N} U_nV_mx^* = UV_mx^*
\]
for all $x^*\in X^*_+$. As $X$ is a KB-space, it is a band in its bi-dual $X^{**}$, see \cite[Theorem 2.4.12]{meyer1991}. Thus, by \cite[Proposition 1.4.15]{meyer1991}, the elements of $X$ are precisely the order continuous linear functionals on $X^*$ whence an adjoint operator on $X^*$ is automatically order continuous. Consequently, for $x^* \in X^*_+$ we have
\[
\sup_{m\in \N} UV_mx^* = UVx^*,
\]
so that altogether we also obtain the inequality $\sup_{n\in\N}U_nV_nx^* \geq UVx^*$. 
\end{proof}

We can now prove the following monotone convergence theorem for positive and contractive $*$-semigroups.

\begin{prop}
\label{prop:supremumsemigroup}
	Let $X$ be a separable KB-space and let $T_n=(T_n(t))_{t>0} \subseteq \cL(X^*,\sigma^*)$ denote an increasing sequence of
	positive and contractive $*$-semigroups with Laplace transforms $(R_n(\lambda))_{\Re \lambda>0}$. Then
	$T(t) \coloneqq \sup_{n\in\N}T_n(t)$ defines a positive and contractive $*$-semigroup whose  Laplace transform $(R(\lambda))_{\Re \lambda>0}$ 
	coincides with $\sup_{n\in\N} R_n(\lambda)$ for all real $\lambda>0$.
\end{prop}
\begin{proof}
	By Lemma \ref{l.increase}(1), $T(t)\coloneqq \sup_{n\in \N} T_n(t)$ defines a positive contraction in $\cL (X^*,\sigma^*)$
	for every $t>0$. By Lemma \ref{l.increase}(2), we find for $t,s>0$ that
	\[ T(t+s) = \sup_{n\in\N} T_n(t+s) = \sup_{n\in\N}T_n(t)T_n(s) = T(t)T(s),\]
	so that $T=(T(t))_{t>0}$ satisfies the semigroup law.
	Since for each $x^* \in X^*_+$ and $x\in X_+$ the function $t\mapsto \applied{T(t)x^*}{x} = \sup_{n\in\N} \applied{T_n(t)x^*}{x}$ is measurable, 
	this shows that $T=(T(t))_{t>0}$ is a $*$-semigroup. Clearly, $T$ is contractive.
	Finally, by monotone convergence
	\[ \sup_{n\in\N} \applied{R_n(\lambda)x^*}{x} = \sup_{n\in\N} \int_0^\infty e^{-\lambda t} \applied{T_n(t)x^*}{x} \, dt = \int_0^\infty e^{-\lambda t}\applied{T(t)x^*}{x} \, dt\]
	for all $\lambda>0$, $x^*\in X^*_+$ and $x\in X_+$. By linearity, this shows that the Laplace transform of $T$ is given by $\sup_{n\in\N} R_n(\lambda)$ for all $\lambda>0$.
\end{proof}

\section{Kernel operators and the strong Feller property}\label{s.sf}

In the previous section we have established  tools that will allow us to prove that a realization $A_\mu$ of our operator $\mathscr{A}$ subject to the nonlocal boundary condition \eqref{eq.bc} generates an injective $*$-semigroup on $L^\infty(\Omega)$ which consists of positive contractions. From the point of view of Markov processes, however, it is more natural to work on the space $B_b(\overline{\Omega})$ of bounded, Borel measurable functions on the set $\overline{\Omega}$. It is particularly important that the involved operators are \emph{kernel operators}, since then we can extract the transition probabilities of the associated stochastic process from these operators. 

In this section, we recall the relevant notions concerning kernel operators. We will also recall the strong Feller property, which is an important tool in studying the asymptotic behavior of transition semigroups of Markov operators. As we will see, the strong Feller property for semigroups also  entails nice continuity properties.

In this section, we set $E\coloneqq \overline{\Omega}$. Note that everything remains valid if $E$ is replaced with a general complete, separable metric space. We denote by $\mathscr{B}(E)$, $B_b(E)$, $C_b(E)$ and $\mathscr{M}(E)$ the Borel $\sigma$-algebra, the space of bounded Borel-measurable functions, the space of bounded continuous functions and the space of singed measures on $E$, respectively.

A \emph{bounded kernel} on $E$ is a map $k: E\times \mathscr{B}(E) \to \C$ such that
\begin{enumerate}
[(i)]
\item the map $x\mapsto k(x,A)$ is Borel-measurable for every $A\in \mathscr{B}(E)$;
\item the map $A\mapsto k(x,A)$ is a complex measure for every $x \in E$ and
\item  $\sup_{x\in E} |k|(x,E) < \infty$, where $|k|(x,\cdot)$ denotes the total variation of the measure
$k(x,\cdot)$.
\end{enumerate}
An operator $K \in \mathscr{L}(B_b(E))$ is called \emph{kernel operator} if there exists a kernel $k$ such that 
\begin{equation}\label{eq.kernelop}
(Kf)(x) = \int_E f(y)\, k(x,dy)
\end{equation}
for every $f\in B_b(E)$ and $x\in E$. As there is at most one kernel $k$ satisfying \eqref{eq.kernelop}, we call $k$ the \emph{kernel associated with $K$} and, conversely, $K$ the \emph{operator associated with k}. Likewise, we can associate an operator $K' \in \mathscr{L}(\mathscr{M}(E))$ with $k$ by setting
\[
K'\nu (A) \coloneqq \int_E k(x, A)\, d\nu(x)
\]
for $A\in \mathscr{B}(E)$. As it turns out, a bounded linear operator $K$ on $B_b(E)$ is a kernel operator if and only if
the norm adjoint $K^*$: $B_b(E)^* \to B_b(E)^*$, defined by $K^*\varphi \coloneqq \varphi\circ K$ for any norm continuous linear functional $\varphi: B_b(E) \to \mathds{C}$, leaves the space $\mathscr{M}(E)$ invariant. In this case we have $K^*|_{\mathscr{M}(E)} = K'$. 
For us, the following characterization is more important. 

\begin{lem}\label{l.kernelop}
Let $K\in \mathscr{L}(B_b(E))$. Then $K$ is a kernel operator if and only if it is \emph{pointwise continuous}, i.e.\
if $(f_n)\subset B_b(E)$ is a bounded sequence converging pointwise to $f\in B_b(E)$, then $Kf_n$ converges pointwise to $Kf$.
\end{lem}

\begin{proof}
If $K$ is pointwise continuous, setting $k(x,A) \coloneqq (K\one_A)(x)$, we see that $k$ is a kernel. By linearity and the density of simple functions in $B_b(E)$ it follows that $k$ is associated with $K$. The converse follows from dominated convergence.
\end{proof}

\begin{defn}
An operator $K \in \mathscr{L}(B_b(E))$ is called \emph{strong Feller operator} if it is a kernel operator and $Kf\in C_b(E)$ for all $f\in B_b(E)$.
\end{defn}

Let us now assume that $(E, \mathscr{B}(E))$ is additionally endowed with a measure $m$ with full support, i.e.\ for every $x\in E$ and $r>0$ we have $m(B_r(x))>0$. This is certainly the case in our intended application, where $E=\overline{\Omega}$ and $m$ is Lebesgue measure on $\Omega$. If $m$ has full support, then two continuous functions which are equal almost everywhere, are equal everywhere. In particular, an element of $L^\infty(E, m)$ may have at most one continuous representative. Suppose now that $\tilde K \in \cL (L^\infty(E, m))$ is such that for every $f\in L^\infty(E, m)$ the image $\tilde K f$ has a continuous representative. In this case, we will say that $\tilde K$ \emph{takes values in $C_b(E)$}. In view of the closed graph theorem, we may consider $\tilde K$ as a bounded operator from $L^\infty (E, m)$ to $C_b(E)$ in that case. 
Let us consider the canonical injection $\iota: B_b(E) \to L^\infty (E, m)$ which maps a bounded, measurable function to its equivalence class modulo equality almost everywhere. If $ \tilde{K}\in \mathscr{L}(L^\infty(E, m))$ takes values in $C_b(E)$, then  $ K\coloneqq  \tilde{K} \circ \iota$ defines a map from  $B_b(E)$ to $C_b(E)$. It is a natural question whether $\tilde K$ is a strong Feller operator. Unfortunately, this is not true without further assumptions, as $\tilde K\circ \iota$ may fail to be a kernel operator, cf.\ \cite[Example 5.4]{akk16}

However, making use of Lemma \ref{l.kernelop}, we easily obtain the following characterization.

\begin{lem}\label{l.kernellinfty}
Let $\tilde K \in \mathscr{L}(L^\infty (E,m))$ take values in $C_b(E)$ and $\iota: B_b(E) \to L^\infty(E, m)$ be as above. Then $\tilde K\circ \iota$ is a kernel operator if and only if for every bounded sequence $(f_n) \subset L^\infty(E,m)$ that converges almost everywhere to $f$ we have $\tilde K f_n(x) \to \tilde K f(x)$ for all $x\in E$.
\end{lem}

Slightly abusing notation, we define the strong Feller property also for operators on $L^\infty(E,m)$.

\begin{defn}
\label{def.sf}
An operator $\tilde K\in \mathscr{L}(L^\infty(E,m))$ is called \emph{strong Feller operator} if
\begin{enumerate}
[(i)]
\item $\tilde K$ takes values in $C_b(E)$;
\item For every bounded sequence $(f_n)\subset L^\infty(E,m)$ converging pointwise almost everywhere to $f$, we have
$\tilde Kf_n \to \tilde Kf$ pointwise.
\end{enumerate}
\end{defn}

In what follows we will not distinguish between strong Feller operators $\tilde K$ on $L^\infty(E,m)$ and the strong Feller operators $K\coloneqq \tilde K\circ \iota$ on $B_b(E)$. In particular, given a strong Feller operator $\tilde K$ on $L^\infty(E, m)$, we can consider the operator $\tilde K ' \in \mathscr{M}(E)$ (which, of course, should be identified with $(\tilde K \circ \iota)'$). 

As it turns out, a strong Feller operator in the sense of Definition \ref{def.sf} is always an adjoint operator. This we prove next.

\begin{lem}\label{l.amadjoint}
Let $K \in \mathscr{L}(L^\infty(E,m))$ be a strong Feller operator. Then $K$ is an adjoint operator.
\end{lem}

\begin{proof}
In view of Lemma \ref{lem:seqweakstarcont} it suffices to prove that $K$ is sequentially weak$^*$-continuous.
Let us fix $x\in E$ and put $\varphi_x(f) \coloneqq Kf(x)$ for $f\in L^\infty(E, m)$. It follows from the closed graph theorem that
$\varphi_x \in L^\infty(E, m)^*$. We now make use of the continuity condition (ii) from Definition \ref{def.sf} to prove that
$\varphi_x(f) = \langle f, g_x\rangle$ for some $g_x \in L^1(E, m)$. To that end, let $(A_n)$ be a sequence of pairwise disjoint Borel subsets of $U$. Then
\[
f_n \coloneqq \one_{\bigcup_{k=1}^n A_k} \uparrow f\coloneqq \one_{\bigcup_{k=1}^\infty A_k}
\]
almost everywhere, whence the  continuity property (ii) implies that $\nu_x (A) \coloneqq \varphi_x(\one_A)$ defines a $\sigma$-additive measure on $\Omega$. If $m(A)=0$, then $\one_A=0$ almost everywhere, whence $K\one_A\equiv 0$.
Thus $\nu_x$ is absolutely continuous with respect to $m$. By the Radon--Nikodym theorem, $\nu_x$ has a density $g_x \in L^1(E, m)$.

Now let a sequence $(f_n) \subset L^\infty(E,m)$ be given with $f_n \weak^* f$. By the above, 
\[
Kf_n(x) = \langle f_n , g_x\rangle \to \langle f, g_x\rangle = Kf(x)\]
for all $x \in E$. In view of the uniform boundedness principle the sequence $(f_n)$ is uniformly bounded
and hence the sequence $Kf_n$ is bounded. It now follows from the dominated convergence theorem that
$Kf_n \weak^* Kf$ as $n\to \infty$. This finishes the proof.
\end{proof}

The importance of the strong Feller property in the study of asymptotic behavior and continuity properties of transition semigroups stems from the following fact.

\begin{lem}\label{l.ultra}
Let $K,L$ be positive strong Feller operators. Then the product $K\cdot L$ is \emph{ultra Feller}, i.e.\ it maps bounded subsets of $B_b(E)$ to equicontinuous subsets of $C_b(E)$.
\end{lem}

\begin{proof}
See \S 1.5 of \cite{revuz}.
\end{proof}

Thus, if $K$ and $L$ are positive strong Feller operators, then it follows from the Arzel\`a--Ascoli theorem that given a bounded sequence $(f_n)$, we can extract a subsequence $(f_{n_k})$ such that $KLf_{n_k}$ converges locally uniformly, i.e.\ with respect to the compact-open topology. In our setting, it is more beneficial to interpret this convergence with respect to another topology.

The \emph{strict topology} $\beta_0$ is defined as follows. We let $\mathscr{F}_0$ be the space of all functions $\varphi$ on $E$ that vanish at infinity, i.e.\ given $\eps >0$ we find a compact set $K\subset E$ such that $|\varphi (x)|\leq \eps$ for all $x\in E\setminus K$. The locally convex topology $\beta_0$ on $C_b(E)$ is defined by the family $\{q_\varphi : \varphi \in \mathscr{F}_0\}$ of seminorms $q_\varphi: f\mapsto \|\varphi f\|_\infty$. 

On norm bounded subsets of $C_b(E)$, $\beta_0$ coincides with the compact-open topology (see \cite[Theorem 2.10.4]{jarchow}); thus if $K$ and $L$ are positive strong Feller operators and $(f_n)$ is a bounded sequence, then $KLf_n$ has a subsequence which converges with respect to $\beta_0$. The main advantage of considering the topology $\beta_0$ instead of the compact open topology is that $\beta_0$ is \emph{consistent with the duality} between $C_b(E)$
and $\mathscr{M}(E)$, i.e.\ we have $(C_b(E), \beta_0)' = \mathscr{M}(E)$. As it turns out, $\beta_0$ is even the Mackey topology of the pair $(C_b(E), \mathscr{M}(E))$, see \cite[Theorems 4.5 and 5.8]{sentilles}, i.e.\ the finest locally convex topology on $C_b(E)$ consistent with the duality. From this one can infer that an operator on $C_b(E)$ is $\beta_0$-continuous if and only if it is $\sigma (C_b(E), \mathscr{M}(E))$-continuous. It is not difficult to see that the latter is the case if and only if the operator is associated with a bounded kernel, see \cite[Proposition 3.5]{kunze2011}
\smallskip

We now obtain the following result about continuity properties of strong Feller semigroups.

\begin{thm}\label{t.sfcont}
Let $(T(t))_{t>0} \subset \mathscr{L}(L^\infty(E, m), \sigma^*)$ be an injective and contractive $*$-semigroup with generator $A$. Assume furthermore that every operator $T(t)$ is a positive 
strong Feller operator. Then the following hold true.
\begin{enumerate}
[(a)]
\item If $(t_n) \subset (0,\infty)$ converges to $t>0$ and $f\in L^\infty(E)$, then $T(t_n)f \to T(t)f$ locally uniformly.
\item For every $f\in L^\infty(E, m)$ the map $(0, \infty)\times E \ni (t,x)\mapsto (T(t)f)(x)$ is continuous.
\end{enumerate}
\end{thm}

\begin{proof}
(a) Note that as a consequence of Lemma \ref{l.ultra} and the semigroup property $T$ consists of ultra Feller operators.
Let $s\coloneqq \inf\{ t_n : n\in \N\} >0$. The sequence $(T(t_n-s)f)$ is bounded and is thus mapped to an equicontinuous set by the ultra Feller operator $T(s)$. Thus $f_n \coloneqq T(t_n)f$ has a subsequence $(f_{n_k})$ which converges with respect to $\beta_0$, say to the function $g\in C_b(E)$. In particular, $f_{n_k}$ converges to $g$ in the weak$^*$-sense in $L^\infty(E,m)$.

We now have
\[
R(\lambda)g = \lim_{k\to\infty}R(\lambda) f_{n_k} = \lim_{k\to \infty} T(t_{n_k}) R(\lambda)f = T(t)R(\lambda )f =
R(\lambda)T(t)f.
\]
Here, the first equality is the weak$^*$-continuity of $R(\lambda)$. The second and the last equality follow from the fact that $R(\lambda)$ commutes with every operator $T(s)$ for $s>0$. The third equality follows from  Corollary \ref{c.strongcont} since $R(\lambda)$ takes values in $D(A)$. As $R(\lambda)$ is injective, we must have $g= T(t)f$.

In the same fashion we see that every subsequence of $T(t_n)f$ has a subsequence which converges (with respect to $\beta_0$) to $T(t)f$. Hence the whole sequence converges.\smallskip

(b) Let $(t_n, x_n) \to (t,x)$. By (a), $T(t_n)f \to T(t)f$ uniformly on the compact set $\{x\}\cup\{ x_n : n\in \N\}$.
\end{proof}

\section{Preliminary results on bounded domains}\label{s.bdd}

We recall that a set $\Omega\subset \R^d$ is called \emph{Dirichlet regular}, if at every point $z\in \partial \Omega$ there exists a \emph{barrier at $z$}, i.e.\ there is a radius $r>0$ and a function $w\in C(\overline{\Omega\cap B_r(z)})$, where
$B_r(z)$ denotes the open ball of radius $r$ centered at $z$, such that
\[
\Delta w \leq 0 \quad\mbox{in } \mathscr{D}(\Omega\cap B_r(z)),\qquad w(z) = 0\quad\mbox{and}\quad
w(x)>0\, \mbox{for}\, x\in \Omega\cap B_r(z).
\]

By classical results, see e.g.\ \cite[Theorem 2.14]{gt}, 
a bounded open set $\Omega$ is Dirichlet regular if and only if the classical Dirichlet problem is well posed, i.e.\ for every $\varphi \in C(\partial \Omega)$ 
we find a harmonic function $u \in C^2(\Omega)\cap C(\overline{\Omega})$ such that $u=\varphi$ on $\partial \Omega$. We should point out that every bounded Lipschitz domain is Dirichlet regular, more generally, every bounded domain that satisfies the uniform exterior cone condition is Dirichlet regular. In $\R$, every open set is Dirichlet regular. In $\R^2$, every open and simply connected subset is Dirichlet regular. For proofs and more information, we refer the reader to \cite{dl}.\smallskip

We will now recall some results from \cite{akk16} concerning diffusion operators subject to nonlocal boundary conditions on \emph{bounded} sets. Throughout, $U$ will be a bounded subset of $\R^d$. We will later apply these results to certain 
subsets $U$ of $\Omega$.

A function $g: \overline{U} \to \R$ is called \emph{Dini-continuous}, if the modulus of continuity
\[
\omega_g(t) \coloneqq \sup\big\{ |g(x) - g(y)| : x,y \in \overline{U}, |x-y|\leq t \big\}
\]
satisfies
\[
\int_0^1 \frac{\omega_g(t)}{t}\, dt < \infty.
\]
Clearly, every H\"older continuous function is Dini-continuous.\medskip

We now recall some results concerning the situation on bounded subsets of $\R^d$ from \cite{akk16}. 
We make the following assumptions.

\begin{hyp}\label{h.bddcoefficients}
Let $U\subset \R^d$ be a bounded, Dirichlet regular set and assume that we are given functions $\alpha_{ij} \in C(\overline{U})$ and $\beta_j \in L^\infty(U)$ which are real-valued for $i,j=1, \ldots, d$. The diffusion coefficients $\alpha_{ij}$ are assumed to be symmetric and strictly elliptic in the sense that there exists a constant $\kappa>0$ such that
for all $x\in \overline{U}$ and $\xi \in \R^d$ we have
\[
\sum_{i,j=1}^d \alpha_{ij}(x)\xi_i\xi_j \geq \kappa |\xi|^2.
\]
Finally, we assume that either
\begin{enumerate}
[(i)]
\item the coefficients $\alpha_{ij}$ are Dini-continuous or
\item $U$ satisfies the uniform exterior cone condition.
\end{enumerate}
\end{hyp}

We will then set 
\[
\mathscr{B}u \coloneqq \sum_{i,j=1}^d \alpha_{ij}D_iD_ju+\sum_{j=1}^d \beta_jD_j u
\]
for $u\in W(U)$.

Also on bounded domains $U$ we  consider a measure-valued function on the boundary which will give us our boundary condition. In contrast to the situation on unbounded domains, we here also allow  sub-probability measures. 
This will be important in our approximation scheme in the next section. We make the following assumptions.

\begin{hyp}\label{h.bddmeasure}
We let $\gamma: \partial U \times\mathscr{B}(U) \to \mathscr{M}(U)$. We will occasionally write $\gamma (z) \coloneqq \gamma (z, \cdot) \in \mathscr{M}(U)$. We assume that
\begin{enumerate}
\item for every $z\in \partial U$ the measure $\gamma(z)$ is positive and satisfies $0\leq \gamma(z, U) \leq 1$;
\item the map $z\mapsto \gamma (z)$ is $\sigma (\mathscr{M}(U), C_b(U))$-continuous.
\end{enumerate}
\end{hyp}

We now define the operator $B$ on $L^\infty(U)$ as follows. We set 
\begin{align*}
D(B)  \coloneqq \Big\{u & \in C_b(\overline{U})\cap W(U) : \mathscr{B}u\in L^\infty(U)\\
& \qquad u(z) = \int_U u(x)\, \gamma (z, dx)\,\forall\, z\in \partial U\Big\}.
\end{align*}

From \cite{akk16} we infer the following properties of the operator $B$.

\begin{prop}\label{p.boundeddomain}
Assume Hypotheses \ref{h.bddcoefficients} and \ref{h.bddmeasure} and let $B$ be defined as above. 
Then the following hold true:
\begin{enumerate}
[(a)]
\item  $(0,\infty) \subset \rho (B)$. For $\lambda >0$ the resolvent $R(\lambda, B) \in \mathscr{L}(L^\infty(U))$ is a positive operator that satisfies $\|\lambda R(\lambda, B)\|\leq 1$.
\item $B$ is the generator of an analytic semigroup $S = (S(t))_{t>0}$ which is positive and contractive.
\item The operators $R(\lambda, B)$ ($\lambda >0$) and $S(t)$ ($t>0$) are strongly Feller in the sense of Definition \ref{def.sf}.
\end{enumerate}
\end{prop}

\begin{proof}
It was seen in \cite[Theorem 4.8]{akk16} that $B$ generates an analytic semigroup on $L^\infty(U)$ which is positive. In \cite[Proposition 4.12]{akk16} it was proved that this semigroup is also contractive. This shows (b). Inspecting the proofs, we see that actually all statements concerning (a) were proved along the way. Part (c)  was established in \cite[Proposition 5.7]{akk16} for $R(\lambda, B)$ and in \cite[Corollary 5.8]{akk16} for the semigroup.
\end{proof}

\begin{rem}
We should point out that we can view the semigroup $S$ generated by $B$ also as a contractive and injective $*$-semigroup on $L^\infty(U)$. Indeed, 
being analytic, the semigroup $S$ is immediately norm continuous and the resolvent can be computed from the semigroup via an $\mathscr{L}(L^\infty(U))$-valued Bochner integral. From this the weaker measurability and integrability conditions in Definition \ref{d.sg} follow. The only thing which is not obvious is that we are dealing with adjoint operators. This, however, follows from Lemma \ref{l.amadjoint} in view of the strong Feller property.
\end{rem}

We now collect some appropriate maximum principles for our situation.

\begin{lem}\label{l.cmp} 
Assume Hypothesis \ref{h.bddcoefficients}.
Let $u\in W(U)$ be a complex valued function such that $\mathscr{B}u$ coincides almost everywhere with a continuous function on $U$ and assume that $|u(x)| \leq |u(x_0)|$ in a neighborhood of $x_0$. Then
\[
\Re [\overline{u(x_0)}\mathscr{B}u(x_0)| \leq 0.
\]
In particular, if $u$ is real valued and $u(x_0)>0$, then $\mathscr{B}u(x_0)\leq 0$.
\end{lem}

\begin{proof}
Lemma 3.2 of \cite{as14}.
\end{proof}

The next Lemma relates the position of possible maxima to the boundary condition. Here, and also subsequently, we use the following notation. Given a measure $\gamma$ on $U$ and a function $u\in C(\overline{U})$, we define $\langle u, \gamma \rangle$ by
\[
\langle u, \gamma\rangle \coloneqq \int_U f\, d\gamma.
\]

\begin{lem}\label{l.bc}
Assume Hypothesis \ref{h.bddmeasure}. Let $u\in C(\overline{U})$ be a real-valued function such that $u(z) \leq \langle u, \gamma(z)\rangle$ for all $z\in \partial U$. If $c\coloneqq \max_{x\in \overline{U}}u(x)>0$, then we find a point $x_0\in U$ such that $u(x_0) = c$.
\end{lem}

\begin{proof}
This is Lemma 4.10 of \cite{akk16}.
\end{proof}

In the proof of \cite[Lemma 4.10]{akk16}, the boundedness of $U$ is only used to infer, that, by compactness, there is some $x_0 \in \overline{U}$ with $u(x_0) = \max_{x\in \overline{U}} u(x)$. Then it is proved that $x_0$ cannot lie on the boundary $\partial U$. However, inspecting the proof, we see that we obtain the following version for \emph{unbounded} domains:

\begin{lem}\label{l.bc2}
Assume Hypothesis \ref{h.mu} and let $u\in C_b(\overline{\Omega})$ be a real-valued function such that $u(z) \leq \langle u, \mu (z)\rangle$ for all $z\in \partial \Omega$. If $S \coloneqq \sup_{x\in \overline{\Omega}} u(x) >0$, then $u(z) < S$ for all $z\in \partial \Omega$.
\end{lem}

We can now establish a maximum principle for our differential operator that involves the boundary condition.

\begin{lem}
\label{lem:maxprinz}
Assume Hypotheses \ref{h.bddcoefficients} and \ref{h.bddmeasure} and let $\lambda >0$. If $u \in C(\overline{U})\cap W(U)$ is a real valued function such that $\mathscr{B} u$ coincides almost everywhere on $U$ with a continuous function,
$(\lambda -\mathscr{B})u \leq 0$ on $U$ and $u \leq \applied{ u}{ \gamma }$ on $\partial U$, then
$u \leq 0$.
\end{lem}

\begin{proof}
If there exists $x \in U$ with $u(x) >0$, then $c\coloneqq \sup_{x\in \overline{U}} u(x) >0$. As 
$u\leq \applied{ u}{ \gamma}$
it follows from Lemma \ref{l.bc} that there is some $x_0 \in U$ with $u(x_0) = c$. 
By Lemma \ref{l.cmp}
we have $\mathscr{B} u(x_0) \leq 0$. Consequently, \[ 0\leq \lambda u(x_0)^2 \leq u(x_0)\mathscr{B} u(x_0) \leq 0, \]
in contradiction to $u(x_0) > 0$. This proves that $u\leq 0$.
\end{proof}

\section{The elliptic equation}\label{s.ell}

We are now ready to tackle  the solvability of the elliptic equation

\begin{equation}\label{eq.ell}
\begin{cases}
\lambda u - \cA u & = f \mbox{ on } \Omega\\
u(z) & = \applied{ u}{ \mu(z)} \mbox{ on } \partial \Omega
\end{cases}
\end{equation}
for $\lambda >0$ and $f\in L^\infty(\Omega)$. From now on, we are again in the situation of Hypotheses \ref{h.coeff} and \ref{h.mu}. In particular, $\Omega$ may be an unbounded set and the coefficients in the operator $\mathscr{A}$ may be unbounded. 

The main idea to construct solutions to Equation \eqref{eq.ell} is the same as in \cite{mpw02}, namely to consider approximate problems on bounded domains and to show that the solutions of these approximate problems converge, in a suitable sense, to a solution of Equation \eqref{eq.ell}. To that end, we set
$\Omega_n \coloneqq \Omega \cap B_{n+1}(0)$, where, as before, $B_r(x)$ denotes the open ball of radius $r$, centered at $x$. Note that as an intersection of two Dirichlet regular sets, the set $\Omega_n$ is again Dirichlet regular, cf.\ \cite[Lemma 3.5]{ab99}. We also recall that in the case where the diffusion coefficients $a_{ij}$ are merely assumed to be continuous, we have explicitly required in (ii) of Hypothesis \ref{h.coeff} that $\Omega_n$ satisfies the uniform outer cone condition.
Altogether, we see that Hypothesis \ref{h.bddcoefficients} is satisfied for $U = \Omega_n$, $\alpha_{ij}=a_{ij}|_{\Omega_n}$
and $\beta_j = b_j|_{\Omega_n}$.\smallskip

To define approximate boundary conditions on $\partial\Omega_n$, we proceed as follows.
We fix functions $\rho_n \in C(\R^d)$ satisfying $\one_{B_n(0)} \leq \rho_n\leq \one_{B_{n+1}(0)}$ and define
$\mu_n\colon \partial  \Omega_n \times \mathscr{B}(\Omega_n) \to \mathscr{M}(\Omega_n)$ by setting
\begin{equation}
\label{eq.mun}
\mu_n (z, A) = \begin{cases}
\rho_n(z) \int_A \rho_n(x)\, \mu (z, dx) & \mbox{ for } z \in \partial \Omega_n\cap\partial\Omega\\
0 & \mbox{ for } z \in \partial\Omega_n \setminus \partial\Omega \subset \partial B_{n+1}(0).
\end{cases}
\end{equation}

As before, we occasionally write $\mu_n(z)\coloneqq \mu_n(z,\cdot)$.  

\begin{lem}\label{l.approxmeas}
For the measures $\mu_n$ defined above, the following assertions hold true:
\begin{enumerate}
[(a)]
\item Every $\mu_n(z)$ is a positive measure satisfying $0 \leq \mu_n(z, \Omega_n) \leq 1$;
\item the map $z \mapsto \mu_n(z)$ is $\sigma (\mathscr{M}(\Omega_n), C_b(\Omega_n))$-continuous;
\item for every $z \in \partial \Omega_n \cap\partial\Omega_{n+1}$ we have $\mu_n(z) \leq \mu_{n+1}(z)$.
\end{enumerate}
\end{lem}

\begin{proof}
(a) This follows directly from the inequalities $0\leq \rho_n \leq 1$ and the fact that every $\mu(z)$ is a probability measure.

(b) Let $(z_k) \subset \partial\Omega_n$ be such that $z_k \to z$. 
If $\abs{z} < n+1$, then also $\abs{z_k} < n+1$ for all but finitely many $k$. We may thus assume that 
$(z_k) \subset \partial\Omega_n\cap\partial\Omega$ converges to $z \in \partial\Omega_n\cap \partial\Omega$. Let $f \in C_b(\Omega_n)$. 
Extending the function $f\cdot \rho_n$ by zero outside $\Omega_n$ we obtain a bounded and continuous function on all of $\Omega$. 
Thus, 
\[ \applied{f}{\mu_n(z_k)}  = \rho_n(z_k)\applied{ f\rho_n}{ \mu(z_k)} \to \rho_n(z) \applied{ f\rho_n}{ \mu(z)} =\applied{ f}{ \mu_n(z)} \]
as $k\to \infty$, by the continuity of $z \mapsto \mu (z)$ and $\rho_n$.

 If, on the other hand,
$\abs{z} = n+1$ then the convergence $\applied{ f}{ \mu_n(z_k)} \to 0 = \applied{ f}{ \mu_n(z)}$ follows from the boundedness of
the integrals $\int f \rho_n \, d\mu (z_k)$ and the fact that $\rho_n(z) \to 0$ as $z \to \partial B_{n+1}(0)$.

(c) This follows immediately from the definition, noting that the functions $\rho_n$ are pointwise increasing.
\end{proof}

It follows that the measures $\gamma=\mu_n$ satisfy Hypothesis \ref{h.bddmeasure}. Thus, we can define the operator
$A_n$ on $L^\infty(\Omega)$ as follows. We set $A_n u = \cA u$ for $u \in D(A_n)$, where
 \begin{align*}
 D(A_n) \coloneqq \Big\{ u &\in C(\overline{\Omega_n}) \cap W(\Omega_n) : \cA u \in L^\infty(\Omega_n),\\
 &\qquad u(z) = \int_{\Omega_n} u(x)\, \mu_n (z, dx)\,\, \forall\, z\in\partial\Omega_n \Big\}
 \end{align*}
It follows from Proposition \ref{p.boundeddomain}, that $(0,\infty) \subset \rho (A_n)$, and for $\lambda >0$ the operator
$R(\lambda, A_n)$ is positive  and satisfies 
$\|\lambda R(\lambda, A_n)\|\leq 1$. Given a function $f\in L^\infty(\Omega)$, we set
$u_n \coloneqq R(\lambda, A_n)f$. Here, in slight abuse of notation, we have identified $f$ with its restriction to $\Omega_n$. We will do so also in what follows. 

Note that $u_n \in D(A_n)$ so that, by the definition of the measure $\mu_n$, we have that $u_n(z) = 0$ for all $z\in \partial\Omega_n \setminus \partial\Omega \subset \partial B_{n+1}(0)$. Thus, setting
$\tilde{u}_n(x) = u_n (x)$ for $x\in \overline{\Omega_n}$ and $\tilde{u}_n(x) = 0$ for $x\in \overline{\Omega}\setminus \overline{\Omega_n}$ we obtain a continuous function on all of $\overline{\Omega}$. In what follows, we will not distinguish between $u_n$ and its extension $\tilde u_n$ to $\overline{\Omega}$.

We will show that the approximative solutions $u_n$ converge to a solution of Problem \eqref{eq.ell}
on the unbounded domain $\Omega$. We prepare this by the following two lemmas in which the fact that $u_n$ is the resolvent of $A_n$ applied to $f$ is not important. We therefore formulate them in greater generality.

\begin{lem}
\label{lem:subsequencesolution}
Assume Hypothesis \ref{h.coeff} and let $u_n \in C(\overline{\Omega_n})\cap W(\Omega_n)$ be a uniformly bounded sequence, say $\norm{u_n}_\infty \leq M$ for all $n\in \N$,
	such that for every $m\in \N$ the sequence $(\cA u_n)_{n\geq m}$ is uniformly bounded on the set $\Omega_m$. We moreover assume that there exists a function $g\colon \Omega \to \R$ such that for every $m\in \N$ the sequence
	$(\cA u_n)_{n\geq m}$ converges pointwise almost everywhere on $\Omega_m$ to $g$. 

	Then $(u_n)$ possesses a subsequence that converges locally uniformly and in $W^{2,p}_\loc(\Omega)$ for all $p\in (1,\infty)$
	to a function $u\in C_b(\Omega)\cap W(\Omega)$ such that $\cA u =g$.
\end{lem}
\begin{proof}
For any $U\Subset \Omega$ we may choose $n_0 \in \N$ such that $U\Subset \Omega_{n_0}$ 
and thus conclude from \cite[Proposition 3.4]{akk16} that there is a constant $C=C(U)$ such that
\begin{align*}
\norm{u_n}_{C^1(U)} &\leq C\big(\norm{u_n - \cA u_n}_{L^\infty(\Omega_{n_0})} + \norm{u_n|_{\partial \Omega_n}}_{C(\partial \Omega_{n_0})} \big)\\
&\leq C\big(2M + \norm{\cA u_n}_{L^\infty(\Omega_{n_0})}\big)<\infty
\end{align*}
for all $n\geq n_0$. By exhausting $\Omega$ with increasing sets $U\Subset \Omega$, 
it follows from the Arzel\`a--Ascoli theorem and a diagonal argument that a subsequence of $u_n$ 
(which, for ease of notation, we denote by $u_n$ again) converges locally uniformly on $\Omega$ to some function $u\in C_b(\Omega)$.
Moreover, given $p\in (1,\infty)$ we obtain from \cite[Theorem 9.11]{gt} that there is a constant $C= C(p, U, \cA, n_0)$ such that
\[
\norm{u_n}_{W^{2,p}(U)} \leq C\big( \norm{\cA u_n}_{L^p(\Omega_{n_0})} + \norm{u_n}_{L^p(\Omega_{n_0})}\big)
\]
holds for all $n\geq n_0$.
Applying this estimate to the difference $u_n-u_m$, it follows from the above by dominated convergence that $(u_n)$ is a Cauchy sequence in $W^{2,p}(U)$.
Since $U$ and $p\in(1,\infty)$ were arbitrary, it now follows that $u \in W(\Omega)$
and that $(u_n)$ along with its first and second derivatives converge in $L^p_\loc(\Omega)$ 
for any $p\in (1,\infty)$. By the structure of $\cA$, this shows that also $\cA u_n \to \cA u$ in $L^p_{\loc}(\Omega)$
and therefore $\cA u = g$.
\end{proof}

Lemma \ref{lem:subsequencesolution} allows us to prove that the solutions of our auxiliary problems converge locally uniformly and in $W^{2,p}_\mathrm{loc}(\Omega)$ to a function in $C_b(\Omega)\cap W(\Omega)$. It is an important question whether one can extend this limit to a continuous function on the closure $\overline{\Omega}$. The next lemma provides a sufficient condition for this.

\begin{lem}
\label{lem:continuous-up-to-the-boundary}
	Let $(u_n) \subset C(\overline{\Omega})$ be a sequence such that $0\leq u_n \leq u_{n+1}$ and $u_n|_{\Omega_n} \in W(\Omega_n)$ for every $n\in\N$.
	Define $u(x) \coloneqq \sup_{n\in\N} u_n(x)$ for $x\in \overline{\Omega}$ 
	and suppose that $u|_\Omega \in C_b(\Omega)\cap W(\Omega)$ and $u|_{\partial \Omega} \in C_b(\partial \Omega)$.
	Finally, assume that  there is $\lambda>0$ such that $\lambda u_n - \cA u_n \leq \lambda u - \cA u$ on $\Omega_n$ for every $n\in\N$. Then $u\in C_b(\overline{\Omega})$.
\end{lem}

\begin{proof}
	As a supremum of continuous functions, $u$ is lower semi-continuous. Since $u$ is assumed to be continuous in 
	$\Omega$, it remains to show continuity of $u$ on $\partial \Omega$. 
	To that end, let $z \in \partial \Omega$ and $(z_n) \subset \overline{\Omega}$ be a sequence converging to $z$.
Pick an index $m\in \N$ such that $z \in \partial \Omega\cap B_m(0)$, so that $\Omega_m$ contains a neighborhood
of $z$.\smallskip 

	As an auxiliary result, let us first show that  we can find a function $\phi \in C(\partial \Omega_m)$ with 
	$\phi \geq u$ on $\partial \Omega_m$ and $\varphi (s) = u(s)$ for all $s \in \partial \Omega_m \cap \overline{B}_m(0)$. We set
	 $\tilde \phi(s) \coloneqq u(s)$ for all $s\in C_1\coloneqq \partial \Omega_m \cap \overline{B}_m(0)$.
	Since $C_1\subset \partial \Omega \cap \overline{B}_m(0)$ and $u\in C_b(\partial \Omega)$, it follows that 
	$\tilde \varphi \in C_b(C_1)$. Now let $C_2 \coloneqq \partial \Omega_m \setminus B_{m+1}(0)$ and $M=\max_{x\in\overline{\Omega_m}}u(x)$. For $s \in C_2$ 
	 we set $\tilde \phi(s) \coloneqq M$, thus obtaining a continuous function $\tilde\varphi$ on the closed set $C_1\cup C_2$.
	 Using the Tietze extension theorem, we can extend $\tilde \phi$ to a continuous function on $\partial \Omega_m$.
	Finally, we define $\phi(s) \coloneqq \max\{\tilde \phi(s), u(s)\}$ for all $s\in \partial \Omega_m$. Clearly, $\phi$
	is continuous in  $s\in \partial\Omega_m \setminus B_{m+1}(0)$, as for such $s$, we find a neighborhood of $s$ contained in $\partial \Omega_m$, where $\phi$ is continuous as maximum of two continuous functions. Moreover, $\phi$ is continuous in $C_2\setminus \partial\Omega$, as there $\phi\equiv\tilde \phi$. It remains to consider $s\in \partial\Omega_m\cap\partial B_{m+1}(0)$. Note that $\phi (s) = M$ for such an $s$. If $u(s) <M$, then $\phi = \tilde\phi$ in a neighborhood of $s$, proving that $\phi$ is continuous in $s$. If, on the other hand, $u(s) = M$, then $u(z) \to u(s)$ as $z\to s$ in $\partial\Omega\cap B_{m+1}(0)$, hence also $\phi (z) \to \phi (s)$ as $z\to s$ in $\partial \Omega\cap B_{m+1}(0)$. But this also hold
	if $z\to s$ in $\partial\Omega_m \setminus B_{m+1}(0)$, as there $\phi \equiv M$. This shows that, altogether, $\phi$ is a continuous function on $\partial\Omega_m$ which, by construction, has all other desired properties.\smallskip

	By \cite[Proposition 3.3]{akk16} we find $w\in C(\overline{\Omega_m})\cap W(\Omega_m)$ such that $\lambda w-\cA w = \lambda u-\cA u$ on $\Omega_m$ and
	$w=\phi$ on $\partial \Omega_m$. 
	Note that for each $n\geq m$ we have $u_n \leq u \leq \phi$ on $\partial \Omega_m$ and
	$\lambda u_n - \cA u_n \leq \lambda u-\cA u$ on $\Omega_m$. It thus follows from Lemma \ref{lem:maxprinz}, applied with $\gamma = 0$, that
	$u_n \leq w$ on $\overline{\Omega_m}$ and hence $u \leq w$  on $\overline{\Omega_m}$.

	Now observe that 
	\[ \limsup_{n\to\infty} u(z_n) \leq \limsup_{n\to\infty} w(z_n) = w(z) = \phi(z) = u(z).\]
	As $z\in \partial \Omega$ was arbitrary, this shows that $u$ is also upper semi-continuous and hence $u\in C(\overline{\Omega})$.
\end{proof}

We can now prove the main result of this section.

\begin{thm}\label{t.elliptic}
For every $\lambda>0$, $(R(\lambda,A_n))_{n\in\N}$ is an increasing sequence of positive operators on $L^\infty(\Omega)$
such that for every $f \in L^\infty(\Omega)$ the sequence $R(\lambda, A_n)f$ converges locally uniformly in 
$\Omega$ to a function $R(\lambda) f \in D_{\mathrm{max}}$ satisfying
$\norm{\lambda R(\lambda)f}_{\infty}\leq \norm{f}_{\infty}$ and $(\lambda - \cA)R(\lambda )f = f$. Moreover, $R(\lambda)$
is a positive operator, i.e.\ for $f\geq 0$ we also have $R(\lambda)f\geq 0$.
\end{thm}

\begin{proof}
Fix $f\in L^\infty(\Omega)$ and $\lambda >0$.
Let us first assume that $f \geq 0$. We consider the approximate operators $A_n$ introduced above and set $u_n \coloneqq R(\lambda, A_n)f$. Extending it by $0$ outside $\Omega_n$, we consider $u_n$ as a continuous function on all of $\overline{\Omega}$. It follows from Proposition \ref{p.boundeddomain} that $u_n \geq 0$ and $\|\lambda u_n\|_\infty\leq \|f\|_\infty$.

We claim that $0\leq u_n \leq u_{n+1}$ on $\Omega_n$. To see this, put $v= u_n - u_{n+1}$. If $z \in \partial \Omega_n$
satisfies $|z|<n+1$, then also $z \in \partial\Omega_{n+1}$. Using that $u_n$ and $u_{n+1}$ satisfy the boundary condition,
we find
\[ v(z) = \applied{u_n}{\mu_n(z)} - \applied{u_{n+1}}{\mu_{n+1}(z)} \leq \applied{u_n}{\mu_n(z)}  - \applied{u_{n+1}}{\mu_n(z)} = \applied{v}{\mu_n(z)}\]
since $u_{n+1}\geq 0$ and $\mu_n(z) \leq \mu_{n+1}(z)$. If $z\in \partial \Omega_n\cap \partial B_{n+1}(0)$, then
$\mu_{n}(z) = 0$ and also in this case we find $v(z) \leq \langle v, \mu_n(z)\rangle$.
Moreover, on $\Omega_n$ we have
\[ \lambda v-\cA v = \lambda u_n - \cA u_n - (\lambda u_{n+1} - \cA u_{n+1}) = f-f = 0\]
almost everywhere, which shows that $\cA v$ possesses the continuous representative $\lambda v$.
Consequently, by Lemma \ref{lem:maxprinz}, we have $v\leq 0$ as claimed.

We now define $u(x) \coloneqq \sup_{n\in\N} u_n(x)$ for $x\in \overline{\Omega}$. Since $\cA u_n  = \lambda u_n - f$ on $\Omega_n$,
it follows that for every $m \in \N$ the sequence $(\cA u_n)_{n\geq m}$ is uniformly bounded on $\Omega_m$ and converges pointwise to $\lambda u-f$. Thus, we conclude from Lemma \ref{lem:subsequencesolution}
that $u\in C_b(\Omega)\cap W(\Omega)$ and that $\cA u = \lambda u-f$.
Now put
\[ \varphi(z) \coloneqq \int_{\Omega} u(x)\mu (z, dx)  \]
for $z\in \partial \Omega$. Since $u\in C_b(\Omega)$, the function $\varphi$ belongs to $C_b(\partial \Omega)$ by the continuity of $z\mapsto \mu(z)$.
Moreover, using monotone convergence, we find
\[
u(z) = \sup_{n\in \N} u_n(z) = \sup_{n\in \N} \applied{ u_n}{ \mu_n(z)} 
= \sup_{n\in\N} \rho_n(z) \applied{ u_n\rho_n}{ \mu (z)} = \applied{ u}{ \mu(z)} = \varphi (z)
\]
for all $z\in \partial\Omega$. In particular, $u$ is continuous on $\partial \Omega$. 
Therefore, we obtain from Lemma \ref{lem:continuous-up-to-the-boundary} that $u\in C_b(\overline{\Omega})$.

Now let $f \in L^\infty(\Omega)$ be real valued. 
We have $f= f^+- f^-$. Then $R(\lambda, A_n)f= R(\lambda, A_n)f^+ - R(\lambda, A_n)f^-$. By the above, $u_n^\pm \coloneqq R(\lambda, A_n)f^\pm$ converges locally uniformly to a function 
$u^\pm \in C(\overline{\Omega}) \cap W(\Omega)$ with $u^\pm(z) = \applied{ u^\pm}{ \mu (z)}$ for all $z\in \partial \Omega$
and $\lambda u^\pm - \cA u^\pm = f^\pm$. Consequently, $R(\lambda, A_n)f$ converges locally uniformly 
to $u \coloneqq u^+-u^-$, which is an element of $D_{\mathrm{max}}$ and solves $\lambda u - \cA u = f$. The case of a complex valued $f$ can be handled similarly, decomposing $f= \Re f + i \Im f$. 
\end{proof}

We next want to define the realization $A_\mu$ of the differential operator $\mathscr{A}$ that appears in Theorem \ref{t.m1}. To that end, we first prove that the operators $R(\lambda)$, constructed in Theorem \ref{t.elliptic}, form a pseudoresolvent.

\begin{lem}\label{l.pseudoresolvent}
Let for $\lambda >0$ the operator $R(\lambda) \in \mathscr{L}(L^\infty(\Omega))$ be given as in Theorem \ref{t.elliptic}.
Then we have:
\begin{enumerate}
[(a)]
\item For $\lambda >0$, the operator $R(\lambda)$ is an adjoint operator.
\item 
The family $(R(\lambda))_{\lambda >0}$ is a pseudoresolvent, i.e.\ we have
\[
R(\lambda_1) - R(\lambda_2) = (\lambda_2 - \lambda_1)R(\lambda_1) R(\lambda_2).
\]
\end{enumerate}
\end{lem}

\begin{proof}
(a) Note that the operators $R(\lambda, A_n)$ are adjoint operators in view of Lemma \ref{l.amadjoint} and part (c) of Proposition \ref{p.boundeddomain}. Now (a) follows from part (1) of Lemma \ref{l.increase}.

(b) We have
$R(\lambda_1, A_n) - R(\lambda_2, A_n) = (\lambda_2 - \lambda_1) R(\lambda_1, A_n)R(\lambda_2, A_n)$ for all $\lambda_1, \lambda_2 >0$.
In view of the definition of the operators $R(\lambda_1)$ and $R(\lambda_2)$, part (b)  follows immediately from Lemma \ref{l.increase}(2).
\end{proof}

Since $(R(\lambda))_{\lambda >0}$ is a pseudoresolvent, the kernel and the range of $R(\lambda)$ are independent of $\lambda>0$. However, as $(\lambda - \mathscr{A})R(\lambda)f=f$, it follows that $\ker R(\lambda ) = \{0\}$ for all $\lambda >0$.  By \cite[Proposition B.6]{arendt2001} $(R(\lambda))_{\lambda>0}$ is the resolvent of an operator. 

\begin{defn}\label{d.am}
The operator $A_\mu$ is defined as the unique operator for which $R(\lambda, A_\mu) = R(\lambda)$ for all $\lambda >0$.
In particular, $D(A_\mu)$ is the range of $R(\lambda)$.
\end{defn}

We can now characterize the domain $D(A_\mu)$ in a different way.

\begin{lem}\label{l.minimal}
Let  $\lambda >0$ be given. For $f\geq0$ the function $R(\lambda, A_\mu)f$ is minimal among the positive solutions
 of $\lambda u - \mathscr{A}u = f$ in $D_\mathrm{max}$.
\end{lem}

\begin{proof}
Let $0\leq u \in D_\mathrm{max}$ be such that $\lambda u - \mathscr{A}u = f$. Given $n \in \N$, let $u_n = R(\lambda, A_n)f$.
Then we have $(\lambda - \mathscr{A})(u_n - u) = 0$ on $\Omega_n$. Moreover, we have for $z \in \partial \Omega_n$
that
\[
u_n(z) - u(z) = \langle u_n, \mu_n(z)\rangle - \langle u, \mu (z)\rangle \leq 
\langle u_n - u, \mu_n (z)\rangle.
\]
By Lemma \ref{lem:maxprinz}, $u_n \leq u$. Taking the supremum over $n$, it follows that $R(\lambda, A_\mu)f 
=\sup_{n\in \N} u_n \leq u$. This finishes the proof.
\end{proof}

Let us now prove that the resolvent of the operator $A_\mu$ consists of strong Feller operators.

\begin{lem}
For $\lambda >0$ the operator $R(\lambda, A_\mu)$ is a strong Feller operator.
\end{lem}

\begin{proof}
As $R(\lambda, A_\mu)$ takes values in $D_\mathrm{max}\subset C_b(\overline{\Omega})$, it only remains to prove the continuity condition (ii) in Definition \ref{def.sf}. As we are dealing with positive operators, it actually suffices to consider merely increasing sequences, cf.\ \cite[Lemma 5.5]{akk16}.

So let an increasing sequence $(f_n)\subset L^\infty(\Omega)$ be given that is uniformly bounded and consists of positive functions. We set $f\coloneqq \sup_{n\in \N} f_n$. We fix $\lambda>0$ and set $u_n \coloneqq R(\lambda,A_\mu)f_n \in \Dmax \subseteq 
	C_b(\overline{\Omega})\cap W(\Omega)$.
	Since $R(\lambda, A_\mu)$ is a positive operator, $(u_n)$ is an increasing and uniformly bounded sequence of positive functions.
	Let $u(x) \coloneqq \sup_{n\in\N} u_n(x)$ for $x\in \overline{\Omega}$.
	Note that $\cA u_n = \lambda u_n -f_n$ is uniformly bounded and converges pointwise almost everywhere to $\lambda u-f$.
	Hence, it follows from Lemma \ref{lem:subsequencesolution} that $u\in C_b(\Omega)\cap W(\Omega)$ and $\lambda u - \cA u =f$.
Consequently, the mapping
\[ \varphi(z) \coloneqq \int_{\Omega} u(x)\mu (z, dx)  \]
defines a function $\phi \in C_b(\partial \Omega)$ by the continuity of $z\mapsto \mu(z)$.
Moreover, by monotone convergence we obtain
\[ u(z) = \sup_{n\in \N} u_n(z) = \sup_{n\in \N} \applied{ u_n}{ \mu(z)} = \applied{ u}{ \mu(z)} = \varphi (z)  \]
for all $z\in \partial\Omega$. In particular, $u$ is continuous on $\partial \Omega$. 
Therefore, we obtain from Lemma \ref{lem:continuous-up-to-the-boundary}  that $u\in C_b(\overline{\Omega})$.
This shows that $u\in \Dmax$ and $\lambda u-\cA u=f$. 

As a consequence of Lemma \ref{l.increase}, $R(\lambda, A_\mu)$ is an adjoint operator, whence it follows that
$R(\lambda, A_\mu)f_n \weak^* R(\lambda, A_\mu)f$. Since $L^1(\Omega)$ separates $C_b(\overline{\Omega})$, we must have
$u = R(\lambda, A_\mu)f$.
\end{proof}

\section{Unique solvability of the elliptic equation}\label{s.lyap}

Throughout this section, we assume Hypotheses \ref{h.coeff} and \ref{h.mu}.
We have seen in Lemma \ref{l.minimal} that for positive $f$ the function $R(\lambda, A_\mu)f$ is the minimal solution
of the elliptic equation $\lambda u - \cA u = f$ in $\Dmax$. It is a natural question, when the elliptic equation 
is uniquely solvable in $\Dmax$, i.e.\ whether $D(A_\mu) = \Dmax$. 
Without further assumptions, this is not the case; see \cite[Example 7.10]{mpw02} for an
example where $\Omega = \R^d$, i.e.\ we do not have boundary conditions.

Let us begin with the following Lemma.

\begin{lem}\label{l.deqdmax}
The following are equivalent:
\begin{enumerate}
[(i)]
\item $\one\in D(A_\mu)$.
\item $D(A_\mu) = \Dmax$.
\end{enumerate}
\end{lem}

\begin{proof}
Let us assume that $\one \in D(A_\mu)$. To prove (ii), we only need to show that for some $\lambda >0$ the operator 
$\lambda - \cA$ is injective 
on $\Dmax$. Indeed, $R(\lambda, A_\mu)$ is a bijection between $D(A_\mu)$ and $L^\infty(\Omega)$ and
$\lambda - \cA : \Dmax \to L^\infty(\Omega)$ is clearly surjective. Thus, if $\lambda -\cA$ is injective on $\Dmax$, then
$R(\lambda, A_\mu)(\lambda -\cA)$ is a bijection from $\Dmax$ to $D(A_\mu)$ and $R(\lambda, A_\mu)(\lambda - \cA) u = u$
for $u \in \Dmax$.

So fix $\lambda >0$ and  let $u\in \Dmax$ with $\lambda u - \cA u = 0$ be given. We assume without loss of generality that $-1\leq u(x) \leq 1$
for all $x \in \overline{\Omega}$. Then $v\coloneqq \one - u$ is a positive function which satisfies $\lambda v -\cA v = \lambda
\one$. As $\one \in D(A_\mu)$, we must have $R(\lambda, A_\mu)\lambda\one = \one$. It follows from
Lemma \ref{l.minimal} that $\one \leq v = \one - u$, i.e.\ $u\leq 0$. Similarly, $\tilde v \coloneqq \one + u$ is a positive function
with $\lambda\tilde v - \cA\tilde v = \lambda\one$ and with the same arguments we find $u\geq 0$. This proves
that $\lambda - \cA$ is injective on $\Dmax$ and finishes the proof of the implication (i) $\Rightarrow$ (ii). The converse implication
is trivial.
\end{proof}

We will see in the next section that $A_\mu$ generates a positive, injective and 
contractive $*$-semigroup $T_\mu$ on $L^\infty(\Omega)$.  Noting that
$\cA \one = 0$, we see that $\one \in D(A_\mu)$ is equivalent to $\one \in \ker A_\mu$ which, in view of Proposition 
\ref{prop:AWF} is equivalent to $T_\mu(t)\one = \one$ for all $t>0$.  Thus, the elliptic equation is uniquely solvable if 
and only if the semigroup generated by $A_\mu$ is Markovian.\medskip

We next provide a sufficient condition for $\lambda -\cA$ to be injective on $\Dmax$. This condition involves 
the existence of a certain \emph{Lyapunov function} for $\cA$.

\begin{hyp}
\label{hyp:Lyapunov}
	There exists a function $V \in C(\overline{\Omega})\cap W(\Omega)$, such that
	\begin{enumerate}[(a)]
		\item $\lim_{\abs{x} \to\infty} V(x) = \infty$.
		\item $\cA V$ coincides almost everywhere on $\Omega$  with a continuous function that is bounded on bounded subsets of $\Omega$.
		\item there is a radius $r>0$ such that $(\lambda - \cA)V \geq 0$ on $\Omega\setminus B_r(0)$.
	\end{enumerate}
\end{hyp}

\begin{lem}
\label{lem:unique}
	Assume Hypothesis \ref{hyp:Lyapunov}. Let $\lambda >0$ and let $u\in W(\Omega)\cap C_b(\overline{\Omega})$ be 
	such that 
	$\cA u$ has a continuous version and such that $(\lambda - \cA)u \leq 0$. Then
	\begin{equation}\label{eq.unique} \sup_{x\in \overline{\Omega}} u(x) \leq \sup_{z\in\partial \Omega} u^+(z).\end{equation}
\end{lem}
\begin{proof}
Note that as a consequence of Hypothesis \ref{hyp:Lyapunov} we may assume that $(\lambda - \cA)V\geq 0$ on $\Omega$, as we may replace $V$ by $V+c\one_{\overline{\Omega}}$ if necessary. We assume this in what follows.

	For each $n\in\N$ define $u_n \coloneqq u - \frac{1}{n}V$ and note that by Hypothesis \ref{hyp:Lyapunov}(a) we may find a constant $C\geq 0$ such that
	$V\geq -C$. Therefore, $u_n \leq u+\frac{1}{n}C$ on $\overline{\Omega}$ for each $n\in\N$ and in particular $u_n$ is bounded from above.
	We immediately obtain from this that 
	\[\sup_{z\in \partial \Omega} u_n^+(z) \leq \sup_{z\in \partial \Omega} u^+(z) + \frac{1}{n} C\]
	for each $n\in\N$. Moreover, since $u_n$ converges to $u$ pointwise, it also follows that 
	\begin{equation}\label{eq.unique2} \lim_{n\to\infty} \sup_{x\in\overline{\Omega}} u_n(x)  = \sup_{x\in\overline{\Omega}} u(x).\end{equation}
	Indeed, given $\eps>0$ we can pick $x_0 \in \overline{\Omega}$ such that $u(x_0) > \sup u - \eps$. Then we may find $n_0 \in \N$ such that
	\[\sup_{x\in\overline{\Omega}}u_n(x) \leq \sup_{x\in\overline{\Omega}}u(x) + \eps \leq u(x_0) +2\eps \leq u_n(x_0) +3\eps \leq \sup_{x\in\overline{\Omega}}u_n(x) + 3\eps \]
	for every $n\geq n_0$, which proves \eqref{eq.unique2} as $\eps>0$ was arbitrary. To show \eqref{eq.unique},
	it thus suffices to show that 
	\begin{align}
	\label{eqn:aimlemmaunique}
	\sup_{x\in\overline{\Omega}}u_n(x) \leq \sup_{z\in \partial\Omega}u_n^+(z)
	\end{align}
	for every $n\in\N$.
	
	It follows from Hypothesis \ref{hyp:Lyapunov}(a) that $\lim_{\abs{x}\to\infty} u_n(x) = -\infty$ for any $n\in\N$. Thus we find $x_n \in \overline{\Omega}$ with $u_n(x_n) = \max_{x\in \overline{\Omega}} u_n(x)$. 
	If $x_n \in \partial \Omega$ then \eqref{eqn:aimlemmaunique} holds true, so assume that $x_n \in \Omega$. As $\cA V$ has a continuous version, so does $\cA u_n$ and we can conclude from Lemma \ref{l.cmp}  that
	$\cA u_n (x_n) \leq 0$. Since both $(\lambda - \cA)V \geq 0$ and $(\lambda - \cA) u \leq 0$ we find $(\lambda -\cA)u_n \leq 0$ and it follows that $\lambda u_n (x_n) \leq \cA u_n(x_n)\leq 0$. Thus, in this case, \eqref{eqn:aimlemmaunique} holds trivially.
\end{proof}

\begin{thm}
	\label{thm:unique}
	Assume Hypothesis \ref{hyp:Lyapunov} and also assume that there exists an index $N  \in \N$ and an $\eps>0$ such that 
	$\mu(z, \Omega_N) \geq \eps$ for all $z\in \partial\Omega$.
	Let $\lambda > 0$ and $u\in \Dmax$ such that $(\lambda -\cA)u\leq 0$. Then $u\leq 0$.
\end{thm}

\begin{rem}
The condition that there exists an $\eps>0$ and an index $N\in \N$ such that $\mu(z, \Omega_N) \geq \eps$ is a mild concentration condition for the measures $\mu(z)$. It is in particular satisfied whenever the set $\{\mu (z) : z\in \partial \Omega \}$ is tight. As the map $z\mapsto \mu (z)$ is $\sigma (\mathscr{M}(\Omega), C_b(\Omega))$-continuous this is in particular the case, whenever $\partial\Omega$ is compact, e.g.\ for an outer domain. However, this condition is weaker than tightness. For example, if $\Omega = (0,\infty)\times \R$, then we might chose for $z=(0,y) \in \partial\Omega$
the measure $\mu(z) = 1/2\delta_{(1,1)}+1/2\delta_{(y,0)}$. These measures satisfy the concentration condition but they are not tight.
\end{rem}

\begin{proof}[Proof of Theorem \ref{thm:unique}]
	Assume to the contrary that $u(x_0)>0$ for some $x_0 \in \overline{\Omega}$. By Lemma \ref{lem:unique}, we have
	\[ \sup_{x\in\overline{\Omega}}u(x) \leq \sup_{z\in\partial\Omega} u^+(z), \]
which implies that $\sup_{z\in\partial\Omega} u^+(z) > 0$. We set $S\coloneqq \sup_{x\in\overline{\Omega}}u(x) = \sup_{z\in \partial\Omega} u(z)>0$.

	We claim that $\sup_{x\in\overline{\Omega_N}}u(x) < S$. Otherwise, we would have $\sup_{x\in \overline{\Omega}} u(x) = u(x_1)$
	for some $x_1 \in \overline{\Omega_N}$. By Lemma \ref{l.bc2}, we must have $x_1\in \Omega_N$. It now follows from Lemma \ref{l.cmp}  that
	\[ \lambda u(x_1) \leq \cA u (x_1) \leq 0, \]
	in contradiction to $u(x_1)>0$.
	
	Thus, we must have that $\sup_{x\in\overline{\Omega_N}} u(x) = \sup_{x\in\Omega_N} u(x) = S-\rho$ for some $0< \rho\leq S$.
	Now  pick a sequence $(z_n) \subset \partial\Omega$ such that $u(z_n) \to S$ as $n\to\infty$.
	Using the boundary conditions, we see that for every $n\in\N$ we have 
	\begin{align*}
		u(z_n) &= \int_\Omega u \, d\mu(z_n) = \int_{\Omega_N} u \, d\mu(z_n) + \int_{\Omega\setminus \Omega_N} u \, d\mu(z_n) \\
		&\leq (S-\rho) \mu(z_n,\Omega_N) + S\mu(z_n,\Omega\setminus \Omega_N)= S- \rho\cdot\mu(z_n,\Omega_N) \leq S-\eps\rho
	\end{align*}
	By taking the limit $n\to\infty$ we obtain the contradiction $S\leq S-\eps\rho$. This shows that we must have $u\leq 0$ on $\overline{\Omega}$.
\end{proof}

\begin{cor}
\label{cor:uniquesolution}
Assume Hypothesis \ref{hyp:Lyapunov} and that there is some $N\in \N$ and $\eps>0$ such that
$\mu(z, \Omega_N)\geq \eps$ for all $z\in \partial\Omega$.  Then
$D(A_\mu) = \Dmax$.
\end{cor}

\begin{proof}
Let $u \in \Dmax$ be such that $\lambda u -\cA u = 0$. It follows from
Theorem \ref{thm:unique} that $u=0$. Thus, $\lambda - \cA u$ is injective on $\Dmax$, whence $D(A_\mu)=\Dmax$.
\end{proof}

We finally determine the kernel of $A_\mu$ in the case where $\Omega$ is additionally connected.

\begin{cor}\label{c.kernel}
Assume Hypothesis \ref{hyp:Lyapunov} and that there exists $N\in \N$ and $\eps>0$ such that $\mu(z, \Omega_N)\geq \eps$ for all $z\in \partial\Omega$. Moreover,
let $\Omega$ be connected. Then $\ker A_\mu = \mathrm{span}\{\one\}$.
\end{cor}

\begin{proof}
If $u \in \Dmax$ satisfies $-\cA u \leq 0$, then either $u$ is constant or $u \leq 0$. This can be proved repeating the proof of Theorem \ref{thm:unique} till the point where we deduced from the assumption
$\sup_{x\in \overline{\Omega_N}} u(x) = \sup_{x \in \overline{\Omega} }u(x)$ that there must be some 
 $x_1 \in \Omega_N$ such that $u(x_1) = \sup_{x \in \overline{\Omega}} u(x)>0$. At this point, the strict maximum principle
 \cite[Theorem 9.6]{gt} implies that $u$ is constant. In the case where $\sup_{x\in \overline{\Omega_N}} < \sup_{x\in \overline{\Omega}}$, the proof can be finished as  that of Theorem \ref{thm:unique}.
\end{proof}

\section{The semigroup}\label{s.tmu}

After our preparation it is now very easy to establish that $A_\mu$ generates a semigroup. Again, we assume Hypotheses \ref{h.coeff} and \ref{h.mu} throughout this section.

\begin{thm}\label{t.generation}
The operator $A_\mu$ generates a positive and contractive $*$-semigroup $T_\mu = (T_\mu(t))_{t>0}$ on $L^\infty (\Omega)$.
\end{thm}

\begin{proof}
Consider again the operators $A_n$ from Section \ref{s.ell}. By Proposition \ref{p.boundeddomain} the operator $A_n$ generates a contractive, positive and holomorphic semigroup $T_n$ on $L^\infty(\Omega_n)$. We have already remarked that we may also view $T_n$ as an injective and contractive $*$-semigroup. Extending $T_n$ and $R(\lambda, A_n)$ (for $\lambda >0)$ by zero outside $\overline{\Omega}_n$, we obtain a (no longer injective) contractive $*$-semigroup with Laplace transform $R(\lambda, A_n)$. By Theorem \ref{t.elliptic} and Proposition \ref{prop:positivesemigroups} the semigroups $T_n$ are increasing. The claim now follows from Proposition \ref{prop:supremumsemigroup}.
\end{proof}

We should point out that in Theorem \ref{t.generation} we only obtain a semigroup on the space $L^\infty(\Omega)$. In that respect, the situation here is very different from that on bounded domains or for the elliptic equation in Section \ref{s.ell} where the operators we obtained always took values in the space of bounded and continuous functions. It was this fact that allowed us to `lift' an operator on $L^\infty(\Omega)$ to a bounded linear operator on $B_b(\overline{\Omega})$. Afterwards, we could use Lemma \ref{l.kernellinfty} to establish that this lifted operator  is a kernel operator.

Our next goal is to prove that we can also lift the operators $T_\mu (t) \in \mathscr{L}(L^\infty(\Omega))$ for $t>0$ to kernel 
operators on $\overline{\Omega}$. To that end, we will use some results concerning order theoretic properties of kernel operators from \cite{gk15}. In particular, we will use the following result which we formulate in the setting used in Section \ref{s.sf}.

\begin{lem}\label{l.sup}
Let $E$ be a complete, separable metric space and let $k_n$ be a sequence of sub-Markovian kernel on $E$, i.e.\ every $k_n$ is a kernel on $E$ such that $k_n (x, \cdot)$ is a positive measure on $\mathscr{B}(E)$ with $0\leq k_n(x, E) \leq 1$ for every $x\in E$. We denote the associated operators on $B_b(E)$ and $\mathscr{M}(E)$ by $K_n$ and $K_n'$ respectively. We put
$k(x,A) \coloneqq \sup_n k_n(x,A)$ for $x\in E$ and $A\in \mathscr{B}(E)$. Then
\begin{enumerate}
[(a)]
\item $k$ is a sub-Markovian kernel on $E$. We denote the associated operators on $B_b(E)$ and $\mathscr{M}(E)$
by $K$ and $K'$ respectively.
\item  $\sup_n K_n = K$ in $\mathscr{L}(B_b(E))$ and $\sup_n K_n' = K'$ in $\mathscr{L}(\mathscr{M}(E))$.
\item  $\sup_n K_n f = Kf$ for every $f\in B_b(E)_+$ and $ \sup_n K_n' \nu = K'\nu$ for every $\nu \in \mathscr{M}(E)_+$.
\end{enumerate}
\end{lem}

\begin{proof}
(a) follows from \cite[Lemma 3.5]{gk15}. 

Note that since the kernels $k_n$ are sub-Markovian, we have
$K_n' \leq I$ for every $n\in \N$. It follows from \cite[Theorem 3.6]{gk15} that $\sup K_n'$ exists in $\mathscr{L}(\mathscr{M}(E))$ and is again a kernel operator. The proof of \cite[Theorem 3.6]{gk15} shows that the kernel associated to $\sup K_n'$ is exactly $k$. There we also see that $\sup_n K_n \nu = K\nu$ for every $\nu \in \mathscr{M}(E)_+$. Thus our assertions in (b) and (c) concerning $K'$ hold true. Let us now note that if $f= \one_A$ is an indicator function, then
\[
Kf= k(\cdot, A) =\sup_n k_n (\cdot, A) = K_n f.
\]
By linearity, the same holds true whenever $f\geq 0$ is a simple function. For a general $f\in B_b(E)_+$, we find, given $\eps >0$ a simple function $g\geq 0$ with $0\leq g\leq f$ and $\|f-g\|_\infty \leq \eps$. Since the kernels $k_n$, thus also $k$, are sub-Markovian, the operators $K_n$ and $K$ are contractions, whence $\|Kf-Kg\|\leq \eps$ and $\|K_nf-K_ng\|\leq \eps$ for all $n\in \N$. Thus
\[
\|Kf - \sup K_n f\| \leq \| Kf - Kg\|  + \|Kg - \sup K_n g\| + \| \sup K_n g - \sup K_n f\| \leq 2\eps.
\]
As $\eps>0$ was arbitrary, this proves the the rest of the assertions.
\end{proof}

We obtain:

\begin{prop}\label{p.kernelop}
There is a family of kernel operators $(K_\mu (t))_{t>0}$, associated to sub-Markovian kernels on $\overline{\Omega}$, 
such that
\begin{enumerate}
[(a)]
\item $K_\mu(t)f$ is a version of $T_\mu (t)\iota (f)$ for every $f\in B_b(\overline{\Omega})$.
\item $K_\mu (t+s) = K_\mu (t)K_\mu (s)$ for all $t,s>0$.
\end{enumerate}
\end{prop}

\begin{proof}
We again consider the semigroups $T_n$, generated by the approximate operators $A_n$, extended to all of $\overline{\Omega}$ by zero. As $T_n(t)$ takes values in $C_b(\overline{\Omega})$, we can consider the operator
$K_n(t) \coloneqq T_n(t)\circ\iota \in \mathscr{L}(B_b(\overline{\Omega}))$ for every $t>0$. By Proposition \ref{p.boundeddomain}(c) these are kernel operators and as a consequence of Theorem \ref{t.elliptic} the sequence is  increasing. It follows from Lemma \ref{l.sup} that $K_\mu (t) \coloneqq \sup_n K_n(t)$ exists in $\mathscr{L}(B_b(\overline{\Omega}))$ and is a kernel operator. 

As $K_\mu (t) f = \sup_n K_n (t)f$ for all $f\in B_b(\overline{\Omega})_+$ by Lemma \ref{l.sup}(c), $K_\mu (t)f$ is a version of $T_\mu (t)\iota(f)$ for all $f\geq 0$. By linearity, this is also true for general $f$, proving (a).

As for (b), first note that for $t,s>0$ and $n\in \N$, we have $K_n(t)K_n(s)\leq K_\mu (t)K_\mu (s)$, whence
\[
K_\mu (t+s) = \sup_n K_n(t+s) = \sup_n K_n(t)K_n(s) \leq K_\mu (t)K_\mu (s).
\]
On the other hand, for $f\geq 0$, the sequence $K_n(s)f$ is bounded and converges pointwise to $K_\mu (s)f$. As $K_\mu (t)$ is a kernel operator, it follows from Lemma \ref{l.kernelop} that $\sup_n K_\mu (t)K_n(s)f = K_\mu (t)K_\mu(s)f$. From this follows $\sup_nK_n(t)K_n(s)f\geq K_\mu(t)K_\mu(s)f$, which proves the other inequality and thus (b).
\end{proof}

\begin{rem}\label{r.sgcbm}
Using the monotone convergence theorem, we see that
\[
\langle R(\lambda, A_\mu) \iota (f), \nu \rangle = \int_0^\infty e^{-\lambda t}\langle K_\mu (t)f, \nu\rangle\, dt
\]
for all $\lambda >0$ and $f\in B_b(\overline{\Omega})_+$, $\nu\in \mathscr{M}(\overline{\Omega})_+$. By linearity, this holds true also for general $f\in B_b(\overline{\Omega})$ and $\nu \in \mathscr{M}(\overline{\Omega})$. 

This shows that $(K_\mu(t))_{t>0}$ defines an integrable semigroup on the norming dual pair $(B_b(\overline{\Omega}), \mathscr{M}(\overline{\Omega}))$ in the sense of \cite[Definition 5.11]{kunze2011}. Its Laplace transform is given by $(R(\lambda, A_\mu)\circ\iota)_{\lambda >0}$ which, of course, is not injective and thus cannot be the resolvent of an operator. However, we may associate a multi-valued generator to the semigroup $K_\mu (t)$. For this multi-valued generator a characterization of the generator similar to Proposition \ref{prop:AWF} remains valid, see \cite[Proposition 5.7]{kunze2011}.
\end{rem}

\begin{thm}\label{t.sfproperty}
If $\one \in D(A_\mu)$, then $T_\mu$ is Markovian and enjoys the strong Feller property. Note that by Lemma \ref{l.deqdmax} the condition $\one\in D(A_\mu)$ is equivalent to $D(A_\mu) = D_{\max}$.
\end{thm}

\begin{proof}
Note that if $\one \in D(A_\mu)$, then $A_\mu\one =0$. As $A_\mu$ is the generator of $T_\mu$, we must have $T_\mu(t)\one = \one$ for all $t>0$ in view of Proposition \ref{prop:AWF}. We should point out that is is an equality almost everywhere. 
However, as explained in Remark \ref{r.sgcbm}, we can apply the corresponding result to the semigroup $(K_\mu (t))_{t>0}$
on $B_b(\overline{\Omega})$ and obtain $K_\mu (t)\one = \one$ everywhere on $\overline{\Omega}$ for every $t>0$.

Now let $0\leq f \leq \one$ be given, so that $K_n(t) f\uparrow K_\mu (t)f$ pointwise. It follows that $K_\mu (t)f$ is lower semi-continuous. On the other hand
\[\one - K_\mu(t)f =  K_\mu(t)(\one-f) = \sup_n K_n(t)(\one-f)\]
is also lower semi-continuous. As $\one$ is continuous, it follows that $K_\mu (t)f$ is upper semi-continuous.

Altogether, we have proved that $K_\mu(t)f$ is continuous whenever $0\leq f\leq 1$. Scaling and decomposing a function into positive and negative part, we see that $K_\mu (t)$ is a strong Feller operator. This finishes the proof.
\end{proof}

\section{Asymptotic behavior}\label{s.asympt}

In this section, we will study the asymptotic behavior of the semigroup $T_\mu$ under the assumption that $\ker A_\mu = \lh\{\one\}$.  We note that Corollary \ref{c.kernel} provides a sufficient condition for this to happen. If $\ker A_\mu = \lh\{\one\}$, then in particular $T_\mu$ is Markovian and enjoys the strong Feller property and we can used recent results (\cite{gg, gk15}) on the asymptotic behavior of such semigroups. Of particular importance are \emph{invariant probability measures} of the semigroup. We recall that a measure $\nu^\star \in \mathscr{M}(\overline{\Omega})$ is called \emph{invariant}, if $T_\mu (t)'\nu^\star = \nu^\star$ for all $t>0$, i.e.\ $\nu^* \in \fix (T_\mu')$. 

\begin{thm}\label{t.asymptotic}
Assume that $\ker A_\mu = \lh\{\one\}$. Then there is at most one invariant probability measure for $T_\mu$. 
If there is an invariant probability measure $\nu^\star$, then we have for $f\in L^\infty(\Omega)$ that 
\[
\lim_{t\to \infty} T_\mu(t) f = \int_{\overline{\Omega}} f\, d\nu^\star \cdot \one 
\]
uniformly on compact subsets of $\overline{\Omega}$ and for $\nu \in \mathscr{M}(\overline{\Omega})$ we have
\[
\lim_{t\to \infty} T_\mu' \nu = \nu (\overline{\Omega})\nu^\star
\]
in total variation norm.
\end{thm}

\begin{proof}
If $\ker A_\mu = \lh\{\one\}$, then in particular $\one\in D(A_\mu)$, so that $T_\mu$ enjoys the strong Feller property by 
Theorem \ref{t.sfproperty}. Moreover, in view of Proposition \ref{prop:AWF}, we have $\fix (T_\mu) = \lh\{\one\}$.
We now have to distinguish the situation where the semigroup $T_\mu$ is weakly ergodic (in the sense of \cite{gk14}) and the situation where it is not weakly erdodic.  As $T_\mu$ enjoys the strong Feller property, we infer from \cite[Theorems 4.4 and 5.7]{gk14} that $T_\mu$ is weakly ergodic if and only if $\fix (T_\mu)'$ separates $\fix (T_\mu)$.

If $\fix (T_\mu')$ separates $\fix (T_\mu)$, then the semigroup is weakly ergodic and it follows from \cite[Theorem 4.4]{gk14} that $\fix (T_\mu)$ separates $\fix (T_\mu')$. As $\fix (T_\mu)$ is one-dimensional, it follows in this situation that
$\fix (T_\mu')$ is also one-dimensional. If, on the other hand, $\fix (T_\mu)'$ does not separate $\fix (T_\mu)$, then we must have $\fix (T_\mu') = \{0\}$. In either case, there can be at most one invariant probability measure. 

Now assume that there is an invariant probability measure $\nu^\star$. Then $T_\mu$ is weakly ergodic with ergodic projection $P = \one\otimes \nu^\star$, i.e.\ $Pf = \int_{\overline{\Omega}} f\, d\nu^\star \cdot \one$. It follows from
\cite[Corollary 3.7]{g14} (a related result can be found in Version 1 of \cite{gg} on the arxiv), that for every $\nu \in \mathscr{M}(\overline{\Omega})$ we have 
$T_\mu'(t)\nu \to P'\nu$ in total variation norm as $t\to \infty$. From this it easily follows that
$T_\mu(t) f \to Pf$ with respect to $\sigma (C_b(\overline{\Omega}), \mathscr{M}(\overline{\Omega}))$ as $t\to \infty$.
However, as $T_\mu$ enjoys the strong Feller property, for every sequence $t_n \to \infty$ the sequence $T_\mu(t_n)f$ has a subsequence which converges with respect to $\beta_0$. But as $T_\mu(t)f \to Pf$ with respect to $\sigma (C_b(\overline{\Omega}), \mathscr{M}(\overline{\Omega}))$, the only possible accumulation point is $Pf$ and we find that
$T_\mu(t)f \to Pf$ with respect to $\beta_0$ and thus also uniformly on compact subsets of $\overline{\Omega}$. 
\end{proof}

To establish the existence of an invariant probability measure again the existence of a suitable Lyapunov function is sufficient. Note, however, that such a Lyapunov function has to satisfy more restrictive assumptions then in Hypothesis \ref{hyp:Lyapunov}. Indeed, if $\Omega = \R^d$ and $\mathscr{A} = \Delta$, the Laplace operator, then $V(x) = |x|^2$ can be used as a Lyapunov function in the sense of Hypothesis \ref{hyp:Lyapunov}. However, there is no invariant probability measure for the heat semigroup on $\R^d$.

In \cite{mpw02}, and also other references, using the Krylov--Bogoliubov theorem, invariant measures are constructed as certain weak accumulation points of Ces\`aro means of the semigroup. In our situation, it is more convenient to work with Abel-means.

\begin{lem}\label{l.abel}
Suppose that $\lambda_n \subset (0,\infty)$ is such that $\lambda_n\downarrow 0$ and there is a probability measure
$\nu$ such that $\lambda_nR(\lambda_n, A_\mu)'\nu$ converges to $\nu^\star$ with respect to the $\sigma (\mathscr{M}(\overline{\Omega}), C_b(\overline{\Omega}))$-topology. Then $\nu^\star$ is an invariant measure for $T_\mu'$.
\end{lem}

\begin{proof}
As $R(\lambda, A_\mu)$ is a strong Feller operator, we may view $R(\lambda, A_\mu)'$ as an operator which is continuous with respect to the $\sigma (\mathscr{M}(\overline{\Omega}), C_b(\overline{\Omega}))$-topology. 
We should note that $R(\lambda, A_\mu)'$
is not necessarily injective, whence it may not be the resolvent of an operator. We may, however, view it as the resolvent of
a multivalued and $\sigma(\mathscr{M}(\overline{\Omega}), C_b(\overline{\Omega}))$-closed operator which we may view as multivalued generator of $T_\mu'$. In slight abuse of notation, we denote this operator by $A_\mu'$.

Let $\nu_n \coloneqq \lambda_nR(\lambda_n, A_\mu)'\nu$. Then $\nu_n \rightharpoonup \nu^\star$. Here, and in what follows,
$\rightharpoonup$ denotes convergence with respect to the $\sigma (\mathscr{M}(\overline{\Omega}), C_b(\overline{\Omega}))$-topology. From the identity $(\lambda_n - A_\mu')R(\lambda_n, A_\mu)'=I$, we obtain
\[
A_\mu' \nu_n = \lambda_n \nu_n - \lambda_n \nu \rightharpoonup 0.
\]
By the closedness of $A_\mu'$, we find $\nu^\star \in D(A_\mu')$ and $A_\mu'\nu^\star=0$. Using \cite[Proposition 5.7]{kunze2011} it follows that $\nu^\star$ is invariant.
\end{proof}

We can now prove a Lyapunov criterion that ensures the existence of an invariant probability measure.

\begin{thm}\label{t.lyapunov}
Assume that $\ker A_\mu = \lh\{\one\}$. Suppose furthermore that 
 there is a function $V \in C(\overline{\Omega})\cap W(\Omega)$ such that
\begin{enumerate}
[(i)]
\item $V \geq 0$ and $V(x) \to \infty$ as $|x|\to \infty$;
\item $\mathscr{A}V$ coincides almost everywhere on $\Omega$ with a continuous function that is bounded on bounded subsets and $\mathscr{A}V(x) \to -\infty$ as $|x|\to \infty$;
\item for every $z\in \partial \Omega$ the function $V$ is integrable with respect to $\mu(z)$ and
for the function $v_0(z)\coloneqq \int_\Omega V(x)\,\mu (z, dx)$ we have $v_0 \leq V$ on $\partial \Omega$.
\end{enumerate}
Then $T_\mu$ has a unique invariant probability measure.
\end{thm}

\begin{proof}
In view of Lemma \ref{l.abel}, it suffices to prove that for some $x_0 \in \Omega$ the set
\[
\{ \lambda R(\lambda, A_\mu)'\delta_{x_0} : 0<\lambda \leq 1\}
\]
is tight.\smallskip

As a first step, let us prove that the function $-\mathscr{A}V$ is integrable with respect to the measure $R(\lambda, A_\mu)'\delta_{x_0}$ whenever $\lambda \in (0,1]$. To that end, let us fix $n_0$ so large that $x_0 \in \Omega_{n_0}$. For $n\geq n_0$, let us put $\tilde f_{n} \coloneqq R(\lambda, A_n)(\lambda- \mathscr{A})V$. Then $\tilde f_n \in D(A_n)$. In particular, $\tilde f_n$ satisfies $\tilde f_n(z) = \langle \tilde f_n, \mu_n(z)\rangle$ for all $z\in \partial \Omega_n$. Now put
$f_n \coloneqq \tilde f_n - \one_{\overline{\Omega_n}}V$. Then $(\lambda - \cA)f_n = 0$ on $\Omega_n$.
Since 
\[
\langle \one_{\overline{\Omega_n}}V, \mu_n (z)\rangle \leq \langle V, \mu (z)\rangle = v_0(z) \leq V(z)
\]
for all $z\in \partial\Omega_n$ we infer that $f_n(z) \leq \langle f_n, \mu_n (z)\rangle$ for all $z\in \partial \Omega_n$. It follows from Lemma \ref{lem:maxprinz} that $f_n \leq 0$ on $\Omega_n$ and thus
\[
-R(\lambda, A_n)\mathscr{A}V \leq R(\lambda, A_n)(\lambda - \mathscr{A})V \leq V
\]
on $\Omega_n$, as $\lambda R(\lambda, A_n)V \geq 0$. 

Now pick $c>0$ such that $c-\cA V \geq 0$, which is possible in view of assumption (ii). Note, that $R(\lambda, A_n)c \leq c\lambda^{-1}$. By monotone convergence, we find 
\[
\int_{\overline{\Omega}}\big( c-\mathscr{A}V\big)\, dR(\lambda, A_\mu)'\delta_{x_0}
= \sup_{n\in \N} \big(R(\lambda, A_n)c- R(\lambda, A_n)\mathscr{A}V(x_0)\big) \leq \frac{c}{\lambda}+V(x_0) .
\]

We can now prove the claimed tightness. To that end, let $\eps>0$ be given. As $\mathscr{A}V(x)  \to -\infty$ as $|x|\to \infty$ we find a radius $r>0$ such that $\mathscr{A} V(x) \leq c -\eps^{-1}$ for all $|x|>r$. Consequently,
$\one_{B_r(0)^c} \leq \eps(c-\mathscr{A}V)$ and hence
\begin{align*}
& (\lambda R(\lambda, A_\mu)'\delta_{x_0})(B_r(0)^c)\\
\leq & \eps \lambda \int_{\overline{\Omega}} (c-\mathscr{A} V) d(R(\lambda, A_\mu)'\delta_{x_0})
\leq \lambda \eps (V(x_0) + \frac{c}{\lambda}) \leq \eps(V(x_0)+c) 
\end{align*}
for all $0<\lambda \leq 1$.
\end{proof}

\begin{rem}\label{rem.lyap}
If $V$ satisfies the assumptions of Theorem \ref{t.lyapunov}, then it also satisfies Hypothesis \ref{hyp:Lyapunov}.
\end{rem}

\section{Examples}\label{sect.example}

In this section, we show how the assumptions of Theorem \ref{t.lyapunov} can be verified in concrete situations. We assume that
\[
\Omega = \mathds{R}^d \setminus \overline{B}(0,1) = \{ x\in \mathds{R}^d : \|x\|> 1\}.
\]
We note that $\partial\Omega$ is compact so that whenever $\mu: \partial\Omega \to \mathscr{M}(\Omega)$ satisfies 
Hypothesis \ref{h.mu} the set
$\{\mu(z): z \in \partial\Omega\}$ is tight. In particular, the concentration condition from Theorem \ref{thm:unique} is automatically fulfilled.

We assume that the coefficients $a_{ij}$ and $b_j$ belong to $C(\overline{\Omega})$ for $i,j=1, \ldots, d$ and satisfy
\begin{equation}\label{eq.behavior}
\lim_{|x|\to \infty} \sum_{j=1}^d \big(a_{jj}(x) + b_j(x) x\big) = - \infty.
\end{equation}
In this situation, the function $V(x) = |x|^2$ satisfies $V(x)\geq 0$, $\lim_{|x|\to\infty} V(x) = \infty$ and
$\lim_{|x|\to \infty} \mathscr{A}V(x) = - \infty$, cf.\ \cite[Corollary 6.4]{mpw02}.

\begin{example}
Condition \eqref{eq.behavior} is for example satisfied in the following situations:
\begin{enumerate}
[(a)]
\item If $a_{ij}(x) = \delta_{ij}$ and $b_j(x) = - x_j$, i.e.\ when 
$\mathscr{A}$ is the \emph{Ornstein--Uhlenbeck operator}
\[
\mathscr{A}u(x) = \Delta u(x) -\langle x, \nabla u(x)\rangle.
\]
\item For operators of the form
\[
\mathscr{A}u(x) = \frac{1}{|x|^\alpha} \Delta u(x) - \langle x, \nabla u(x)\rangle,
\]
where $\alpha>0$ (recall that $|x|>1$ for $x\in \Omega$).
\item For operators of the form
\[
\mathscr{A}u(x) = |x|^\alpha\Delta u(x) - |x|^{\beta-1}\langle x, \nabla u(x)\rangle,
\]
where  $\alpha>0$ and $\beta > \alpha- 1$.
\end{enumerate}
\end{example}

\begin{cor}\label{c1}
Let $\Omega$ be as above and assume that the continuous coefficients $a_{ij}$ and $b_j$ satisfy besides Hypothesis \ref{h.coeff} also Condition \eqref{eq.behavior}. Then $D(A_\mu) = D_\mathrm{max}$ and the semigroup $T_\mu$ is
Markovian and enjoys the strong Feller property.
\end{cor}

\begin{proof}
It follows from Condition \eqref{eq.behavior} that the function $V(x) = |x|^2$ satisfies Hypothesis 6.3. Since $\partial \Omega$ is compact and thus $\{\mu(z): z\in \partial\Omega\}$ is tight, the other condition of Theorem \ref{thm:unique} is satisfied and $D(A_\mu) = D_\mathrm{max}$ follows from Corollary \ref{cor:uniquesolution}. The assertions concerning the semigroup $T_\mu$ now follow from Theorem \ref{t.sfproperty}.
\end{proof}

Let us now turn to the existence of an invariant measure. If Condition \eqref{eq.behavior} is satisfied, then
$V(x) = |x|^2$ satisfies condition (i) and (ii) in Theorem \ref{t.lyapunov}. Condition (iii), however, is not satisfied in general
by this function. Indeed, $V$ need not be integrable with respect to the measures $\mu (z)$, for example if $d=1$ and
$\mu(z)$ has a density of the form $c|x|^{-2}$ with respect to Lebesgue measure. Even if $V(x)$ is integrable with respect to all measures $\mu(z)$, we cannot expect that for $z\in \partial \Omega$ we have
$\int V(x)\, \mu (z, dx) \leq 1=V(z)$. However, sometimes we may modify the function $V$ such that this is the case.

\begin{cor}
Let $\Omega$ be as above and assume that the continuous coefficients $a_{ij}$ and $b_j$ satisfy besides Hypothesis \ref{h.coeff} also Condition \eqref{eq.behavior}. Moreover, assume that 
\[
\sup_{|z|=1} \int_\Omega |x|^2\, \mu (z, dx) < \infty.
\]
Then the semigroup $T_\mu$ has a unique invariant measure.
\end{cor}

\begin{proof}
We note that as $\one \in D(A_\mu)$ by Corollary \ref{c1}, the semigroup $T_\mu$ can at most have one invariant probability measure. To prove existence of an invariant probability measure, we show that we can modify the function $V(x) = |x|^2$ in such a way, that the assumptions of Theorem \ref{t.lyapunov} are satisfied.\smallskip

We set $M\coloneqq \sup_{|z|=1} \int_\Omega |x|^2\, \mu (z, dx)$. We claim that we can find $\eps\in (0,1)$ 
such for the set
$S_\eps \coloneqq B_{1+\eps}(0)\setminus B_1(0)$ we have $\mu (z, S_\eps) \leq (1+2M)^{-1}$ for all $z\in \partial \Omega$. To see this, pick a continuous function $f_n: \overline{\Omega} \to [0,1]$ such that $f_n(x) = 1$ for $1 \leq |x| \leq 1+n^{-1}$ and $f_n(x) = 0$ for $|x|\geq 1+2n^{-1}$. Then $f_n\downarrow 0$ pointwise on $\Omega$. By dominated convergence, we have that $\langle f_n, \mu (z)\rangle \downarrow 0$ for every $z\in \partial \Omega$. Since
the function $z\mapsto \langle f_n, \mu(z)\rangle$ is continuous, it follows from Dini's theorem that this convergence is uniform on the compact set $\partial \Omega$. Consequently, we find an index $n$ such that $\langle f_n, \mu(z)\rangle \leq (1+M)^{-1}$ for \emph{all} $z\in \partial \Omega$ and we may put $\eps = 2n^{-1}$.\smallskip

We now pick a function $\varphi \in C^2([1,\infty))$ such that 
\[
\varphi (t) \begin{cases}
= M+1 & \mbox{for } t= 1\\
\in [0,M+1] & \mbox{for } t\in (0, \eps)\\
\in [0, t] & \mbox{for }  t \in [\eps, 1]\\
= t & \mbox{for } t> 1
\end{cases}
\]
and set $\tilde V (x) = \varphi (|x|^2)$. Then $\tilde V$ is a $C^2$-function such that for $|x|>1$ we have 
$\tilde V(x) = V(x)$ and $\mathscr{A}\tilde V(x) = \mathscr{A}V(x)$. In particular, we have $\tilde V(x) \to \infty$ for
$|x|\to \infty$ and $\mathscr{A}\tilde V(x) \to -\infty$ for $|x|\to \infty$, so that conditions (i) and (ii) in Theorem 
\ref{t.lyapunov} are fulfilled. Moreover, we have
\begin{align*}
\int_\Omega \tilde V(x)\, \mu (z, dx) & \leq \int_{S_\eps} M+1\, \mu (z, dx) + \int_{\Omega\setminus S_\eps} |x|^2\, \mu(z, dx)\\
& \leq 
\frac{1+M}{1+2M} + M \leq M+1 = \tilde V(z).
\end{align*}
This proves that $\tilde V$ also satisfies condition (iii) in Theorem \ref{t.lyapunov} so that the existence of an invariant measure follows from that theorem.
\end{proof}

\section*{Acknowledgment}

The author is grateful to Moritz Gerlach and Jochen Gl\"uck for several fruitful discussions concerning this article and their own article \cite{gg}. These discussions in particular resulted in abstract formulation of Lemma \ref{l.increase} and Proposition \ref{prop:supremumsemigroup} on duals of KB spaces.

\end{document}